\documentclass[a4paper,10pt,reqno]{amsart}%{article}

\usepackage{amsmath}
\usepackage{amsfonts}
\usepackage{amsthm}
\usepackage{amssymb}
\usepackage{ulem}
\usepackage[top=2.5cm, bottom=2.5cm, left=2.5cm, right=2.5cm, includefoot]{geometry}

\usepackage[dvipdfmx,setpagesize=false,colorlinks=true,bookmarksnumbered=true]{hyperref}
\usepackage[dvipdfmx]{graphicx,color}

\newtheorem{Def}{Definition}[section]
\newtheorem{Thm}[Def]{Theorem}
\newtheorem{Lem}[Def]{Lemma}
\newtheorem{Assumption}[Def]{Assumption}
\newtheorem{Rem}[Def]{Remark}
\newtheorem{Cor}[Def]{Corollary}

\makeatletter
    
    \@addtoreset{equation}{section}
  \makeatother

\newcommand{\argmin}{\operatornamewithlimits {argmin}}
\allowdisplaybreaks

\usepackage{ascmac}

\newcommand{\mcf}{\mathcal{F}}

\newcommand{\mbbr}{\mathbb{R}}

\newcommand{\al}{\alpha}
\newcommand{\del}{\delta}
\newcommand{\ep}{\epsilon} 
\newcommand{\vp}{\varphi}

\newcommand{\sig}{\sigma}
\newcommand{\D}{\Delta}
\newcommand{\Sig}{\Sigma}
\newcommand{\gam}{\gamma}
\newcommand{\Gam}{\Gamma}
\newcommand{\p}{\partial}

\newcommand{\cip}{\xrightarrow{p}} % <- Convergence in probability
\def\ds#1{\displaystyle{#1}}
\def\nn{\nonumber}

\def\sumj{\sum_{j=1}^{n}}

\title[Two-step estimation of ergodic L\'{e}vy driven SDE]{
Two-step estimation of ergodic L\'{e}vy driven SDE
}
\date{\today}

\author{Hiroki Masuda}
\address[Hiroki Masuda]{Faculty of Mathematics, Kyushu University. 744 Motooka Nishi-ku Fukuoka 819-0395, Japan}
\email{hiroki@math.kyushu-u.ac.jp}
\author{Yuma Uehara}
\address[Yuma Uehara, corresponding author]{Graduate School of Mathematics, Kyushu University. 
744 Motooka Nishi-ku Fukuoka 819-0395, Japan}
\email{ma214003@math.kyushu-u.ac.jp}

\begin{document}

\maketitle

\begin{abstract}
We consider high frequency samples from ergodic L\'{e}vy driven stochastic differential equation 
(SDE) with drift coefficient $a(x,\alpha)$ and scale coefficient $c(x,\gamma)$ involving unknown parameters $\alpha$ and $\gamma$. 
We suppose that the L\'{e}vy measure $\nu_{0}$, has all order moments but is not fully specified. 
We will prove the joint asymptotic normality of some estimators of $\alpha$, $\gamma$ 
and a class of functional parameter $\int\varphi(z)\nu_0(dz)$, which are 
constructed in a two-step manner: first, we use the Gaussian quasi-likelihood for estimation of $(\al,\gam)$; 
and then, for estimating $\int\varphi(z)\nu_0(dz)$ we make use of the method of moments 
based on the Euler-type residual with the the previously obtained quasi-likelihood estimator.
\end{abstract}

\section{Introduction}

It is widely recognized that a diffusion model is a typical candidate model to describe the high activity time-varying dynamics.
However, especially in the biological, technological and financial application, there do exist many phenomena where driving noise process 
exhibits highly non-Gaussian behavior. A jump-type L\'{e}vy process may serve as a suitable building block in modeling such phenomena. 
In this paper, we consider a high frequency data $(X_{t_0},X_{t_1},\dots,X_{t_n})$ from the one-dimensional 
L\'{e}vy driven stochastic differential equation (SDE):
\begin{equation}\label{Model}
dX_t= a(X_t,\alpha)dt+ c(X_{t-}, \gamma)dJ_t,\quad X_0=x_0,
\end{equation}
where:
\begin{itemize}
\item $\alpha=(\al_{l})$ and $\gamma=(\gam_{l})$ 
are unknown finite dimensional parameters and we suppose that each of them are elements of bounded convex domains $\Theta_\alpha\subset\mathbb{R}^{p_\alpha},\Theta_\gamma\subset\mathbb{R}^{p_\gamma}$ and we write $\Theta=\Theta_\alpha\times\Theta_\gamma$ and $p_\alpha+p_\gamma=p$.
\item 
The functional forms of the drift coefficient $a:\mathbb{R}\times\Theta_\alpha\to\mathbb{R}$ and the scale coefficient $c:\mathbb{R}\times\Theta_\gamma\to\mathbb{R}$ are known.
\item $J_t$ is a one-dimensional pure jump L\'{e}vy process with L\'{e}vy measure $\nu_0$. 
\end{itemize}
%For any $\theta=(\al,\gam)\in\Theta$, 
We denote by 
%$P_\theta$ the image measure of the corresponding solution $X$ and for simplicity, we write 
$P_0$ the true image measure of $X$ associated with the true value $\theta_0\in\Theta$. 
Note that we do not consider the case of misspecification of the functional form of the coefficients. 
We suppose that the path of $X_t$ is not observed continuously but observed discretely at high frequency: we consider the samples $(X_{t_0},X_{t_1},\dots,X_{t_n})$, where $t_j=t_j^n=jh_n$ for some $h_n>0$ which satisfies that
\begin{equation}\nn
nh_n^2\to0 \quad  \mbox{and}   \quad nh_n^{1+\epsilon_0}\to\infty,
\end{equation}
for $n\to\infty$ and some $\epsilon_0\in(0,1)$. 
The objective of this paper is to estimate $\theta_0$ and the functional parameter $\int\varphi(z)\nu_0(dz)$ for some function $\varphi$ 
in a two-step manner. 
It is not essential in our results that $X$ has no Wiener part, 
but the absence is assumed from the very beginning just for simplicity of the statements; 
see Remark \ref{hm:rem_wiener.part} for a brief discussion.

Up to the present, many results about the estimation of the diffusion process (this process corresponds to the case of replacing $J_t$ with a standard Wiener process in \eqref{Model}) have been established both continuous sampling case and discrete sampling case.
In the continuous sampling case, the explicit form of its likelihood is given (see, for example, \cite{Liptser2001}).
Hence we can construct the maximum likelihood estimator of $\alpha$ and under some conditions, it has consistency and asymptotic normality (for details, see \cite{Kutoyants2004} and \cite{Rao1999}).
In the discrete sampling case, we can not obtain the closed form of its likelihood in general, so that we have to consider another method. 
Typically, we resort to the quasi-likelihood based on the local Gaussian approximation. 
By the It\^{o}-Taylor expansion, \cite{Kessler1997} gives the estimation scheme in the case of $nh_n\to\infty$ and $nh_n^q\to0$ ($\forall q \geq 2$). 
\cite{Gobet2002} shows its local asymptotic normality; he also shows the local asymptotic normality in the non-ergodic case. 
Needless to say, there are many estimation methods besides (quasi) maximum likelihood method (see, for example, \cite{Kutoyants2004} and \cite{Rao1999}).
We emphasize that these estimation methods essentially rely on the scaling and finite-moment properties of Wiener process. 

Construction of an estimator of $\int\varphi(z)\nu_0(dz)$ is important in the statistical inference associated with L\'{e}vy process.
Recall that the class of bounded continuous functions vanishing in a neighborhood of the origin completely characterizes $\nu_{0}$ \cite[Theorem 8.7]{Sato1999}.
In particular, the parameter $\int\varphi(z)\nu_0(dz)$ corresponds to the $q$th cumulant of $J_1$ for $\varphi(z)=z^{q}$ with $q>2$, 
and also to the cumulant transform of $J_{1}$ for $\varphi(z)=e^{iuz}-1-iuz$, $u\in\mbbr$, 
which is important to assessing the ruin probability in a jump-type L\'{e}vy risk model. 
The example of moment-fitting estimation of $\int\varphi(z)\nu_0(dz)$ from the discretely samples, $(J_{h_n},J_{2h_n},\dots,J_{nh_n})$, are proposed in \cite{Lopez2009} and \cite{Shimizu2009}. 
The main claim of \cite{Lopez2009} says that 
under some moment conditions, for a function $\varphi$ vanishing in a neighborhood of the origin it follows that
\begin{equation}\label{ETEL}
\sqrt{nh_n}\left(\frac{1}{nh_n}\sum_{j=1}^n\varphi(\Delta_j J)-\int\varphi(z)\nu_0(dz)\right)
\overset{\mathcal{L}}\longrightarrow
\mathcal{N}\left(0,\int\varphi(z)^2\nu_0(dz)\right),
\end{equation} 
where $\Delta_j J=J_{jh_n}-J_{(j-1)h_n}$.
However, in the estimation of L\'{e}vy driven SDE, we encounter the difficulty, that is, $(J_{h_n},J_{2h_n},\dots,J_{nh_n})$ cannot be observed directly. 
One may think of utilizing a martingale estimating function for joint estimation of $\theta_0$ and $\int\varphi(z)\nu_0(dz)$. 
However, we then have to specify what kind of conditional expectation is to be used in an explicit way, 
which inevitably requires more specific structural assumptions about $\nu_{0}(dz)$ beyond Assumption \ref{Moments}. 

Here we will take another route. 
Previously, \cite{Masuda2013} used the Gaussian quasi-likelihood, which can apply to a large class of L\'{e}vy processes, making it possible to construct Gaussian quasi maximum likelihood estimators (GQMLE) $\hat{\theta}_n=(\hat{\alpha}_n,\hat{\gamma}_n)$ of the true value 
$\theta_0=(\alpha_0,\gamma_0)$ without any specific information about the noise distribution; also, 
\cite[Theorem 2.7]{Masuda2013} shows that it has consistency and asymptotic normality with rate $\sqrt{nh_n}$. 
By making use of the GQMLE and the functional-parameter moment fitting, we will propose a two-step procedure 
for joint estimation of $\theta$ and $\int\varphi(z)\nu_0(dz)$: 
we first estimate $\alpha$ and $\gamma$ by GQMLE, and next construct the estimator of $\int\varphi(z)\nu_0(dz)$ based on Euler-Maruyama approximation. 
We still do not presume the closed form of the noise distribution, 
so that our way of estimation is beneficial in terms of the robustness against noise misspecification. 
Further the proposed two-step procedure enables us to bypass simultaneous optimization problem, which may result in high computational load. 

%The moment convergence is crucial for detecting the asymptotic behavior of statistics which can be used, for example, derivation of information criteria, mean bias correction  and investigation of mean squared prediction error. 
%Recall that convergence in law does not always imply $L_p$-convergence.
%Under $X_n\overset{\mathcal{L}}\longrightarrow X$, by \cite[Theorem 2.20]{VanderVaart2000}, we have $E[f(X_n)]\longrightarrow E[f(X)]$ if and only if $f(X_n)$ is uniformly asymptotically integrable.
%Based on the notion of the polynomial type large deviation inequality (PLDI) introduced by \cite{Yoshida2011}, \cite[Theorem 2.7]{Masuda2013} gives moment convergence of GQMLE for any polynomial growth function $f$: 
%\begin{equation}
%E[f(\sqrt{nh_n}(\hat{\theta}_n-\theta_0))]\longrightarrow\int_{\mathbb{R}^p} f(u)\phi(u;0,\Sigma) du,
%\end{equation}
%where $\phi(u;0,\Sigma)$ is the probability density function of $\mathcal{N}(0,\Sigma)$ and $\Sigma$ is the asymptotic variance of $\sqrt{nh_n}(\hat{\theta}_n-\theta_0)$.

The organization of this paper is as follows. 
In Section \ref{Notations and Assumptions}, we will introduce notations and assumptions for our main results. 
Section \ref{Main results} provides our main results: the stochastic expansion
\begin{align*}
&\sqrt{nh_n}\left(\frac{1}{nh_n}\sum_{j=1}^n\varphi\left(\frac{X_{jh_n}-X_{(j-1)h_n}-a(X_{(j-1)h_n},\hat{\alpha}_n)}{c(X_{(j-1)h_n},\hat{\gamma}_n)}\right)-\int\varphi(z)\nu_0(dz)\right)\\
&=\sqrt{nh_n}\left(\frac{1}{nh_n}\sum_{j=1}^n\varphi(\Delta_j J)-\int\varphi(z)\nu_0(dz)\right)
+\hat{b}_{n}\sqrt{nh_n}(\hat{\gamma}_n-\gamma_0)+o_p(1),
\end{align*}
and the asymptotic normality of our estimators; 
see \eqref{SE_b} for the explicit form of $\hat{b}_{n}$. 
In particular, the second term of the right-hand side reflects the effect of plugging-in the $\sqrt{nh_{n}}$-consistent 
estimator $\hat{\gam}_{n}$ into the scale components of the Euler-residual sequence.
All the proofs of our main results are presented in Section \ref{Appendix}.

%%%%%
%%%%%

\section{Notations and Assumptions}\label{Notations and Assumptions}

\subsection{Notations}
We denote by $(\Omega,\mathcal{F}, (\mathcal{F}_t )_{t\in\mathbb{R}_+} , \mathbb{P} )$ a complete filtered probability space on which the process X is defined, the initial variable $X_0$ being $\mathcal{F}_0$-measurable and $J_t$ being $\mathcal{F}_t$-adapted and independent of $X_0$.

For abbreviation, we introduce some notations.
\begin{itemize}
\item $E_0[\cdot]$ denotes the expectation operator with respect to $P_0$ and we abbreviate $\int \varphi(z) \nu_0(dz)$ to $\nu_0(\varphi)$.
\item For differentiable function $f$, $\partial_x f$ stands for the derivative with respect to any variable $x$ and $\partial f$ represents the vector of the derivatives of the components of $f$.
\item $t_j:=jh_n$.
\item $E^{j-1}[\cdot]$ stands for the conditional expectation with respect to $\mathcal{F}_{t_{j-1}}$. 
\item $\Delta_j Z$ stands for $Z_{t_j}-Z_{t_{j-1}}$ for any process $Z$.
\item $ \sum_j:=\sum_{j=1}^n$ and $\int_j:=\int_{t_{j-1}}^{t_j}$.
\item $\eta(x,\theta):=a(x,\alpha)c^{-1}(x,\gamma)$ and $M(x,\theta):=\partial_\alpha a(x,\alpha)c^{-2}(x,\gamma)$.
\item 
$f_s:=f(X_s,\theta_0)$ for any function $f$ on $\mathbb{R}\times\Theta$; 
e.g. $a_{t}(\al)=a(X_{t},\al)$ and $M_{t}(\theta)=M(X_{t},\theta)$.
\item We will write $x_n\lesssim y_n$ when there exists a positive constant $C$ such that $x_n\leq Cy_n$ for large enough $n$; $C$ does not depend on $n$ and varies line to line.
\end{itemize}

We define the random functions $G_n^\alpha(\theta)\in\mathbb{R}^{p_\alpha}$ and $G_n^\gamma(\theta)\in\mathbb{R}^{p_\gamma}$ by 
\begin{align*}
G_n^\alpha(\theta)&=\frac{1}{nh_n}\sum_jM_{t_{j-1}}(\theta) (\Delta_jX -h_n a_{t_{j-1}}(\alpha)),\\
G_n^\gamma(\theta)&=\frac{1}{nh_n}\sum_j\left\{\left[-\partial_\gamma c_{t_{j-1}}^{-2}(\gamma)\right](\Delta_jX-h_n a_{t_{j-1}}(\alpha))^2-h_n \frac{\partial_\gamma c_{t_{j-1}}^2(\gamma)}{c_{t_{j-1}}^2(\gamma)}\right\},
\end{align*} 
and the corresponding GQMLE (\cite{Masuda2013}) by
\begin{equation}\nn
\displaystyle \hat{\theta}_n := \argmin_{\theta\in\bar{\Theta}} |(G_n^\alpha(\theta),G_n^\gamma(\theta))|,
\end{equation}
where $\bar{\Theta}$ denotes the closure of $\Theta$ and $|\cdot|$ the Euclidean norm. 

We introduce additional notations associated with GQMLE.
\begin{itemize}
\item $\hat{f}_s:=f(X_s,\hat{\theta}_n)$ for any function $f$ on $\mathbb{R}\times\Theta$; 
for notational brevity, we also use the notation $\partial_{\theta}\hat{f}_{j-1}$ instead of $\widehat{\partial_{\theta}f_{j-1}}$.
\item $\delta_j :=c_{t_{j-1}}^{-1}(\Delta_jX-h_na_{t_{j-1}})$ and $\hat{\delta}_j :=\hat{c}_{t_{j-1}}^{-1}(\Delta_jX-h_n\hat{a}_{t_{j-1}})$.
\item $\hat{v}_n:=\sqrt{nh_n}(\hat{\theta}_n-\theta_0)$ and $\hat{w}_n:=\sqrt{nh_n}(\hat{\gamma}_n-\gamma_0)$.
\end{itemize}

\subsection{Assumptions}
For our asymptotic results, we introduce some assumptions.

\begin{Assumption}[Sampling design]\label{Sampling design}
$nh_n^2\to0$ and $nh_n^{1+\epsilon_0}\to\infty$ for $\epsilon_0 \in (0,1)$. 
\end{Assumption} 

\begin{Assumption}[Moments]\label{Moments}
We have $E[J_1]=0, E[J_1^2]=1$ and $E[|J_1|^q]<\infty$ for all $q>0$.
\end{Assumption}

Although we only assume the moment conditions on $J_1$, the first and the third formulae are valid for all $t>0$, see \cite[Theorem 25.18]{Sato1999} and we have $E[J_t^2]=t$ from the expression of characteristic function of $J_t$.
Further, by the definition of L\'{e}vy measure and the fact that $E[|J_t|^q]$ exists if and only if $\int_{|z|\geq 1} |z|^q \nu_0(dz)$ (see \cite[Theorem 25.3]{Sato1999}), we see that $\int |z|^q \nu_0(dz) < \infty$, for all $q\geq2$ under Assumption \ref{Moments}.

\begin{Assumption}[Smoothness]\label{Smoothness}
\begin{enumerate}
\item The drift coefficient $a(\cdot,\alpha_0)$ and the scale coefficient $c(\cdot,\gamma_0)$ are Lipschitz continuous.
\item For each $i \in \left\{0,1,2\right\}$ and $k \in \left\{0,1,\dots,5\right\}$, the following conditions hold:
\begin{itemize}
%\item The coefficient $(a,c)$ has the extension in $\mathcal{C}(\mathbb{R}\times\bar{\Theta})$ and has partial derivatives $(\partial_x^i \partial_\alpha^k a, \partial_x^i \partial_\gamma^k c)$ which admit the extensions in $\mathcal{C}(\mathbb{R}\times\bar{\Theta})$.
\item 
The coefficient $a(x,\al)$ and $c(x,\gam)$ have partial derivatives 
$\partial_x^i \partial_\alpha^k a(x,\al)$ and $\partial_x^i \partial_\gamma^k c(x,\gam)$, 
and all the functions 
$\al\mapsto \partial_x^i \partial_\alpha^k a(x,\al)$ and $\gam\mapsto\partial_x^i \partial_\gamma^k c(x,\gam)$ 
for each $x\in\mbbr$ (including $\al\mapsto a(x,\al)$ and $\gam\mapsto c(x,\gam)$ themselves) can be continuously extended to the boundary of $\Theta$.

\item There exists nonnegative constant $C_{(i,k)}$ satisfying
\begin{equation}\label{polynomial}
\sup_{(x,\alpha,\gamma) \in \mathbb{R} \times \Theta_\alpha \times \Theta_\gamma}\frac{1}{1+|x|^{C_{(i,k)}}}\left\{|\partial_x^i\partial_\alpha^ka(x,\alpha)|+|\partial_x^i\partial_\gamma^kc(x,\gamma)|+|c^{-1}(x,\gamma)|\right\}<\infty.
\end{equation}
\end{itemize}
\end{enumerate}
\end{Assumption}

In this paper we will assume that $X$ is exponentially ergodic together with the boundedness of moments of any order. 
Let $P_t$ denote the transition probability of $X$.
Given a function $\rho:\mathbb{R}\to\mathbb{R}^+$ and a signed measure $m$ on one-dimensional Borel space, we define
\begin{equation}\nn
||m||_\rho=\sup\left\{|m(f)|:\mbox{$f$ is $\mathbb{R}$-valued, $m$-measurable and satisfies $|f|\leq\rho$}\right\}.
\end{equation}

\begin{Assumption}[Stability]\label{Stability}
\begin{enumerate}
\item
There exists a probability measure $\pi_0$ such that for every $q>0$ we can find positive constants $a$ and $c$ for which 
\begin{equation}\label{Ergodicity}
\sup_{t\in\mathbb{R}_{+}} e^{at} ||P_t(x,\cdot)-\pi_0(\cdot)||_g \le cg(x), \quad x\in\mathbb{R},
\end{equation}
where $g(x):=1+|x|^q$.
\item 
For all $q>0$, we have
\begin{equation}\nn
\sup_{t\in\mathbb{R}_{+}}E_0[|X_t|^q]<\infty. 
\end{equation}
\end{enumerate}
\end{Assumption}

The condition \eqref{Ergodicity} corresponds to the exponential ergodicity when $g$ is replaced by $1$.
When some boundedness conditions about coefficients and their derivatives are assumed, moment conditions written in above can be weakened (see \cite[Section 5]{Masuda2013} for easy sufficient conditions for Assumption \ref{Stability}).

Let $G_\infty(\theta):=(G_\infty^\alpha(\theta), G_\infty^\gamma(\gamma))\in\mathbb{R}^p$ define by 
\begin{align*}
G_\infty^\alpha(\theta)=\int\frac{\partial_\alpha a(x,\alpha)}{c^2(x,\gamma)}(a(x,\alpha_0)-a(x,\alpha))\pi_0(dx),\\
G_\infty^\gamma(\theta)=2\int\frac{\partial_\gamma c(x,\gamma)}{c^3(x,\gamma)}(c^2(x,\gamma_0)-c^2(x,\gamma))\pi_0(dx).
\end{align*}
We need to impose some conditions on $G_\infty(\theta)$ for the consistency of $\alpha$ and $\gamma$.
The sufficient condition for the consistency of general M(or Z)-estimator is given in \cite{VanderVaart2000}.

\begin{Assumption}[Identifiability]\label{Identifiability}
There exist nonnegative constants $\chi_\alpha$ and $\chi_\gamma$ such that
\begin{equation}\nn
|G_\infty^\alpha(\theta)| \geq \chi_\alpha|\alpha-\alpha_0|,\quad |G_\infty^\gamma(\theta)| \geq \chi_\gamma|\gamma-\gamma_0| \quad \mbox{for all} \ \theta.
\end{equation}
\end{Assumption}

Define $\mathcal{I}(\theta_0):=$diag$\left\{\mathcal{I}^\alpha(\theta_0),\mathcal{I}^\gamma(\theta_0)\right\}\in\mathbb{R}_{p}\otimes\mathbb{R}_{p}$ by
\begin{align*}
&\mathcal{I}^\alpha(\theta_0)=\int\frac{(\partial_\alpha a(x,\alpha_0))^{\otimes2}}{c^2(x,\gamma_0)} \pi_0(dx),\\
&\mathcal{I}^\gamma(\theta_0)=4\int\frac{(\partial_\gamma c(x,\gamma_0))^{\otimes 2}}{c^2(x,\gamma_0)} \pi_0(dx),
\end{align*}
where $x^{\otimes2}:=xx^T$ for any vector or matrix $x$ and $T$ means the transpose.
The matrix $\mathcal{I}(\theta_0)$ plays a role like a Fisher-information like quantity in GQML estimation.

\begin{Assumption}[Nondegeneracy]\label{Nondegeneracy}
$\mathcal{I}^\alpha(\theta_0)$ and $\mathcal{I}^\gamma(\theta_0)$ are invertible.
\end{Assumption}

Our estimation of $\nu_0(\varphi)$ will be based on \eqref{ETEL}.
Here we only think of Euclidean space valued $\vp$, while treatment of complex $\vp$ being completely analogous. 
In our setting, we only observe high frequency sample $(X_h,X_{2h_n},\dots,X_{nh_n})$, hence we need to approximate $\Delta_jJ$ to estimate $\nu_0(\varphi)$.
Let $\mathcal{A}$ denote the formal infinitesimal generator with respect to L\'{e}vy process $J$, that is,
\begin{equation}\label{IG}
\mathcal{A}\varphi(x)=\int (\varphi(x+z)-\varphi(x)-\partial\varphi(x)z) \nu_0(dz),
\end{equation} 
for any $\varphi$ such that the integral exists. In what follows we fix a positive integer $q$. 
%Let us provide a set of easy-to-verify sufficient conditions for Assumption \ref{Moment fitting function} and 
%the last condition of Corollary \ref{JAN2}. 
We now define a positive constant $\rho$ fulfilling that
\begin{equation*}
\rho>(1-\epsilon_0)\vee\beta,
\end{equation*}
where $\epsilon_0$ is the same as in Assumption \ref{Sampling design} and $\beta$ denotes the Blumenthal-Getoor index of $J$ defined by
\begin{equation*} 
\beta=\inf\left\{\gamma\geq0;~\int_{|z|\le 1}|z|^\gamma\nu_0(dz)\right\}.
\end{equation*}
Denote by $\mathcal{K}$ the set of all $\mbbr^{q}$-valued functions on $\mbbr$ 
such that its element $f=(f_{k})_{k=1}^{q}:~\mbbr\to\mbbr^{q}$ satisfies the following conditions:
\begin{enumerate}
\item $f$ is five times differentiable.
\item There exist nonnegative constants $C_i$ ($0\le i\le 5$) such that
\begin{align*}
&\limsup_{z\to0}\left\{\frac{1}{|z|^\rho}|f(z)|+\frac{1}{|z|}|\partial f(z)|\right\}<\infty, \nn\\
& \limsup_{z\to\infty}\left\{\frac{1}{1+|z|^{C_0}}|f(z)|+\frac{1}{|z|^{1+C_1}}|\partial f(z)|\right\}<\infty,\\
&\sup_{z\in\mathbb{R}}\frac{1}{1+|z|^{C_i}}|\partial^i f(z)|<\infty,\quad i\in\left\{2,3,4,5\right\}.
\end{align*}
\end{enumerate}

%Define the two function spaces:
%\begin{align*}
%\mathcal{K}_1
%&=\Biggl\{f=(f_{k}):\mbbr\to\mbbr^{q}\, \Bigg|\, \text{$f$ is of class $C^{2}$}, \quad 
%\sup_{0\le s\le h_{n}}E_0\left[|\mathcal{A}^2 f(J_s)|\right]=O(1),
%\nonumber \\
%&{}\qquad \text{and}\quad
%\max_{i=0,1}\int_0^{h_n}\int E_0\left[|\mathcal{A}^i f(J_{s-}+z)-\mathcal{A}^i f(J_{s-})|^2\right]\nu_0(dz) ds=O(1)\Biggr\},
%\nonumber\\
%\mathcal{K}_{2}
%&=\Biggl\{f=(f_{k}):\mbbr\to\mbbr^{q}\, \Bigg|\, \text{$f$ is of class $C^{2}$}, \quad 
%\frac{1}{h_n}\max_{1 \leq j\leq n}E_0\left[\left|\partial f(\delta_j)\right|^2\right]=O(1),
%\nonumber \\
%& {}\qquad \frac{1}{h_n}\max_{1\leq j\leq n} \sup_{u\in[0,1]}E_0\left[\left|\partial f(\Delta_jJ+u(\delta_j-\Delta_jJ))\right|^2\right]
%=O(1),
%\\
%&{}\qquad\text{and} \quad 
%\forall K>0,\quad 
%\max_{1 \leq j\leq n}\sup_{u\in[0,1]}E_0 
%\left[\left|\partial^2 f(\hat{\delta}_j+u(\delta_j-\hat{\delta}_j))\right|^K\right]
%=O(1)
%\Biggr\}.
%\end{align*}

%For our results, we impose some conditions on a moment-fitting function $\varphi$. 

We now impose
\begin{Assumption}[Moment-fitting function]\label{Moment fitting function} $\varphi\in\mathcal{K}$.
%$\varphi(0)=0$, $\varphi'(0)=0$ and $\varphi, \zeta\in\mathcal{K}_1\cap\mathcal{K}_{2}$.
\end{Assumption}
Then, according to the definition of Blumenthal-Getoor index and Assumption \ref{Moments} we have $\nu_0(\varphi)<\infty$.

%Assumption \ref{Moment fitting function} can be readily verified 
%if $\vp$ and its derivatives admit appropriate polynomial majorants. 
%We will provide much simpler conditions later.

%%%%%
%%%%%

\section{Main results}\label{Main results}
The Euler-Maruyama approximation says that
\begin{equation}\nn
X_{t_j} \approx X_{t_{j-1}}+h_n a_{t_{j-1}}+c_{t_{j-1}}\Delta_jJ.
\end{equation}
This suggests that we may formally regard $\delta_j$ as the estimator of $\Delta_jJ$, and indeed it will turn out to be true under our assumptions.
Also, we will see that the Euler residual  $\hat{\delta}_j$, which is constructed only by $(X_{h_n},X_{2h_n},\dots,X_{nh_n})$, may also serve as an estimator of $\Delta_jJ$ (see the proof of Theorem \ref{Bias correction}).

Let
\begin{align*}
& u_n:=\sqrt{nh_n}\bigg(\frac{1}{nh_n}\sum_{j} \varphi(\Delta_jJ) -\nu_0(\varphi)\bigg),
\\
&\hat{u}_n:=\sqrt{nh_n}\bigg(\frac{1}{nh_n}\sum_{j}\varphi(\hat{\delta}_j)-\nu_0(\varphi)\bigg).
\end{align*}
As was mentioned in the introduction, we know that $u_{n}$ is asymptotically normally distributed: 
$u_{n}\overset{\mathcal{L}} \longrightarrow \mathcal{N}(0,\nu_{0}(\vp^{\otimes 2}))$. 
Let
\begin{equation}
\zeta(z) := z\p\vp(z).
\nonumber
\end{equation}

The next theorem clarifies the effect of using the statistics $\hat{\delta}_j$ instead of the unobservable variables $\Delta_j J$.

\begin{Thm}\label{Bias correction}Under Assumptions \ref{Sampling design}-\ref{Moment fitting function}, we have
\begin{equation}\label{SE}
\hat{u}_n=u_n+\hat{b}_n[\hat{w}_n]+o_p(1),
\end{equation}
where $\hat{b}_n \in \mathbb{R}^q\otimes\mathbb{R}^{p_\gamma}$ is defined by
\begin{equation}\label{SE_b}
\hat{b}_n = -\bigg(\frac{1}{nh_{n}}\sum_j\zeta(\hat{\delta}_j)\bigg)\otimes
\bigg(\frac{1}{n}\sum_j\frac{\p_{\gam}\hat{c}_{t_{j-1}}}{\hat{c}_{t_{j-1}}}\bigg).
\end{equation}
\end{Thm}

Building on the stochastic expansion \eqref{SE}, we will see 
that substituting $\hat{\del}_{j}$ into $\D_{j}J$ leads to the different asymptotic covariance matrix of 
the estimator of $\nu_0(\varphi)$. See the comments after Corollary \ref{JAN2} for more details.

\begin{Rem}
Although GQMLE is adopted as the estimator of $\theta_0$, \eqref{SE} is valid for any estimator $\hat{\theta}_n$ which satisfies $E[|\sqrt{nh_n}(\hat{\theta}_n-\theta_0)|^q]<\infty$ for all $q>0$ (cf. the proof of Theorem \ref{Bias correction}).
\end{Rem}

\medskip

We define the estimating function $G_n(\theta)$ for $(\theta, \nu_{0}(\varphi))$ by
\begin{equation}\nn
G_n(\theta)=\bigg( \frac{1}{\sqrt{nh_{n}}}u_n,\, G_n^\alpha(\theta),\, G_n^\gamma(\theta)\bigg),
\end{equation}
where $G_n^\alpha(\theta)$ and $G_n^\gamma(\theta)$ are defined in the previous section. 
Introduce 
\begin{equation}\nn
\Sigma =\begin{pmatrix}
\Sigma_{11}&\Sigma_{12}\\
\Sigma_{12}^T&\Sigma_{22}
\end{pmatrix}
\end{equation}
with $\Sigma_{11}\in\mathbb{R}^q\otimes\mathbb{R}^q$, 
$\Sigma_{12}=(\Sigma_{12,kl})_{k,l}\in\mathbb{R}^q\otimes\mathbb{R}^p$ 
and $\Sigma_{22}=(\Sigma_{22,kl})_{k,l}\in\mathbb{R}^p\otimes\mathbb{R}^p$, 
where
\begin{align*}
\Sigma_{11}&=\nu_0(\varphi^{\otimes2}),\\
\Sigma_{12,kl}&=\begin{cases}
\ds{\int \varphi_{k}(z)z \nu_0(dz) \int \frac{\partial_{\alpha_l}a(x,\alpha_0)}{c(x,\gamma_0)} \pi_0(dx)} &(1 \leq l \leq p_\alpha),\\[3mm]
\ds{2\int \varphi_{k}(z)z^2 \nu_0(dz)  \int \frac{\partial_{\gamma_l} c(x,\gamma_0)}{c(x,\gamma_0)} \pi_0(dx)} &(p_\alpha+1 \leq l \leq p), 
\end{cases}\\
\Sigma_{22,kl}&=\begin{cases}
\ds{\int\frac{\partial_{\alpha_k} a(x,\alpha_0) \partial_{\alpha_l} a(x,\alpha_0)}{c^2(x,\gamma_0)} \pi_0(dx)} 
&(k,l \in \left\{1,\dots, p_\alpha \right\}),\\[3mm]
\ds{4\int\frac{\partial_{\gamma_k} c(x,\gamma_0)\partial_{\gamma_l} c(x,\gamma_0)}{c^2(x,\gamma_0)} \pi_0(dx) \int z^4 \nu_0(dz)}
&(k,l \in \left\{p_\alpha +1,\dots,p \right\}),\\[3mm]
\ds{2\int\frac{\partial_{\alpha_k} a(x,\alpha_0)\partial_{\gamma_l} c(x,\gamma_0)}{c^2(x,\gamma_0)} \pi_0(dx) \int z^3 \nu_0(dz)}
&(k \in \left\{1,\dots, p_\alpha \right\},l \in \left\{p_\alpha +1,\dots ,p \right\}).
\end{cases}
\end{align*}

\begin{Thm}\label{JAN1}If Assumptions \ref{Sampling design}-\ref{Identifiability} and Assumption \ref{Moment fitting function} hold, and if $\Sig$ is positive definite, then
%if the mappings $z\mapsto z\vp(z),\,z^{2}\vp(z)$ belong to $\mathcal{K}_{1}\cap\mathcal{K}_{2}$, 
\begin{equation}\nn
\sqrt{nh_n}G_n(\theta_0) \overset{\mathcal{L}} \longrightarrow \mathcal{N}_{p+q}(0,\Sigma).
\end{equation}
\end{Thm}

\begin{Rem}\label{rem_mconv}
The moment convergence of the estimator 
is crucial for detecting the asymptotic behavior of statistics which can be used, for example, derivation of information criteria, mean bias correction and investigation of mean squared prediction error; see the references cited in \cite{Masuda2013}. 
As for the GQMLE $\hat{\theta}_n$, under Assumption \ref{Sampling design}-\ref{Nondegeneracy} we can deduce
\begin{equation}\nn
E[f(\hat{v}_n)]\longrightarrow\int_{\mathbb{R}^p}f(u)\phi(u;0,\mathcal{I}(\theta_0)^{-1}\Sigma_{22}(\mathcal{I}(\theta_0)^{-1})^T)du,
\end{equation}
for every continuous function $f:\mathbb{R}^p\to\mathbb{R}$ of at most polynomial growth: see \cite[Theorem 2.7]{Masuda2013}. 
In this paper, we do not go into details of the moment convergence of $f(\hat{u}_n)$.
\end{Rem}

%Now, in order to construct a consistent estimator of the asymptotic variance $\Sigma$, 
%which depends on the true value $(\theta_0,\nu_0(\varphi))$, we introduce an additional class of $\varphi$:
%\begin{align*}
%\mathcal{K}_3
%&=
%\Biggl\{f=(f_{k}):\mbbr\to\mbbr^{q}\, \Bigg|\, \text{$f$ is of class $C^{1}$}, \quad 
%\max_{1\leq j\leq n} \sup_{u\in[0,1]}E_0 \left[\left|\partial f(\Delta_jJ+u(\delta_j-\Delta_jJ))\right|^2\right]=o(1),
%\nonumber \\
%&{}\qquad \text{and}\quad 
%\frac{1}{nh_{n}^{2}}\max_{1\leq j\leq n} \sup_{u\in[0,1]}
%E_0 \left[\left|\partial f(\hat{\delta}_j+u(\delta_j-\hat{\delta}_j))\right|^2\right]=o(1)\Biggr\}.
%\end{align*}

Define the statistics $\hat{\Gamma}_n\in\mathbb{R}^{p+q}\otimes\mathbb{R}^{p+q}$ by
\begin{equation}\nn
\hat{\Gamma}_n= \begin{pmatrix}I_q&-\hat{B}_n\\O&-\partial_\theta (G_{n}^{\al}, G_n^{\gam})(\hat{\theta}_n)\end{pmatrix}, 
\end{equation}
where $\hat{B}_n=\begin{pmatrix}O&\hat{b}_n\\\end{pmatrix}\in\mathbb{R}^q\otimes\mathbb{R}^p$.
We also define
\begin{equation}\nn
\hat{\Sigma}_n=\begin{pmatrix}\hat{\Sigma}_{11,n}&\hat{\Sigma}_{12,n}\\
\hat{\Sigma}_{12,n}^T&\hat{\Sigma}_{22,n}\end{pmatrix},
\end{equation}
with $\hat{\Sigma}_{11,n}\in\mathbb{R}^q\otimes\mathbb{R}^q$, $(\hat{\Sigma}_{12,n,kl})_{k,l}\in\mathbb{R}^q\otimes\mathbb{R}^p$ and $(\hat{\Sigma}_{22,n,kl})_{k,l}\in\mathbb{R}^p\otimes\mathbb{R}^p$, where
\begin{align*}
\hat{\Sigma}_{11,n}&=\frac{1}{nh_n}\sum_{j}\varphi^{\otimes2}(\hat{\delta}_j),\\
\hat{\Sigma}_{12,n,kl}&=\begin{cases}
\ds{\bigg(\frac{1}{nh_n}\sum_{j}\varphi_k(\hat{\delta}_j)\hat{\delta}_j\bigg)\bigg(\frac{1}{n}\sum_{j}\frac{\partial_{\alpha_l}\hat{a}_{t_{j-1}}}{\hat{c}_{t_{j-1}}}\bigg)} &(1 \leq l \leq p_\alpha),\\[3mm]
\ds{\bigg(\frac{2}{nh_n}\sum_{j}\varphi_k(\hat{\delta}_j)\hat{\delta}^2_j\bigg)\bigg(\frac{1}{n}\sum_{j}\frac{\partial_{\gamma_l}\hat{c}_{t_{j-1}}}{\hat{c}_{t_{j-1}}}\bigg)} &(p_\alpha+1 \leq l \leq p), 
\end{cases}\\
\hat{\Sigma}_{22,n,kl}&=\begin{cases}
\ds{\frac{1}{n}\sum_{j}\frac{\partial_{\alpha_k}\hat{a}_{t_{j-1}}\partial_{\alpha_l}\hat{a}_{t_{j-1}}}{\hat{c}_{t_{j-1}}^2}}
&(k,l \in \left\{1,\dots, p_\alpha \right\}),\\[3mm]
\ds{\bigg(\frac{4}{n}\sum_{j}\frac{\partial_{\gamma_k}\hat{c}_{t_{j-1}}\partial_{\gamma_l}\hat{c}_{t_{j-1}}}{\hat{c}_{t_{j-1}}^2}\bigg)\bigg(\frac{1}{nh_n}\sum_{j}\hat{\delta}^4_j\bigg)} &(k,l \in \left\{p_\alpha +1,\dots,p \right\}),\\[3mm]
\ds{\bigg(\frac{2}{n}\sum_{j}\frac{\partial_{\alpha_k}\hat{a}_{t_{j-1}}\partial_{\gamma_l}\hat{c}_{t_{j-1}}}{\hat{c}_{t_{j-1}}^2}\bigg)\bigg(\frac{1}{nh_n}\sum_{j}\hat{\delta}^3_j\bigg)} &(k \in \left\{1,\dots, p_\alpha \right\},l \in \left\{p_\alpha +1,\dots ,p \right\}).
\end{cases}
\end{align*}
It will turn out that $\hat{\Sigma}_n$ is a consistent estimator of the asymptotic variance $\Sigma$, 
which depends on the true value $(\theta_0,\nu_0(\varphi))$ under our assumption.

By use of Theorem \ref{Bias correction} and Theorem \ref{JAN1}, we can derive the asymptotic normality of the statistics $(\hat{u}_n,\hat{v}_n)$ only constructed from the observed data $(X_{h_n},X_{2h_n},\dots,X_{nh_n})$.

\begin{Cor}\label{JAN2}
Suppose that Assumptions \ref{Sampling design}-\ref{Moment fitting function} hold and that $\Sig$ is positive definite.
% and that the mappings $z\mapsto z\varphi(z),z^2\varphi(z),\varphi_{1}\vp(z),\dots,\vp_{q}\vp(z)$ 
%belong to $\mathcal{K}_1\cap\mathcal{K}_3$.
Then $\hat{\Sig}_{n}\cip\Sig$ and
\begin{equation}\label{hm:eq2}
\hat{\Sigma}_n^{-1/2}\hat{\Gamma}_n\begin{pmatrix}\hat{u}_n\\ 
\hat{v}_n\end{pmatrix}\overset{\mathcal{L}}\longrightarrow\mathcal{N}(0, I_{p+q}),
\end{equation}
where $I_p$ denotes the $p\times p$ identity matrix.
\end{Cor}

%The asymptotic normality of $u_n$ and $\hat{w}_n$ will follows from 
%the stochastic expansions $u_{n}=\frac{1}{\sqrt{nh_{n}}}\sum_{j}e_{j}+o_p(1)$ 
%and $\hat{w}_n = \frac{1}{\sqrt{nh_{n}}}\sum_{j}\chi_{j} + o_{p}(1)$; see the proof of Lemma \ref{AN} and Lemma \ref{AN3} 
%for the explicit form of $\left(e_j\right)_{j\in\left\{1,\dots,n\right\}}$ and $\left(\chi_j\right)_{j\in\left\{1,\dots,n\right\}}$).

By means of Lemma \ref{AN2} and Lemma \ref{K3}, we can observe that 
\begin{equation}
\hat{b}_n \overset{P_0}\longrightarrow 
b_{0}:=-\left(\int\zeta(z)\nu_0(dz)\right)\otimes\left(\int\frac{\partial_\gamma c(x,\gamma_0)}{c(x,\gamma_0)}\pi_0(dx)\right).
\nonumber
\end{equation}
%This, combined with \eqref{SE} and the expression of $u_n$ and $\hat{w}_n$, gives 
%the asymptotic normality of $\frac{1}{nh_n}\sum_j\varphi(\hat{\delta}_j$): 
%\begin{equation}
%\sqrt{nh_n}\bigg(\frac{1}{nh_n}\sum_j\varphi(\hat{\delta}_j)-\nu_0(\varphi)\bigg)
%= \frac{1}{\sqrt{nh_{n}}}\sum_{j}(e_{j} + b_{0}\chi_{j}) +o_p(1).
%\nonumber
%\end{equation}
Put $B_0:=\begin{pmatrix}O&b_0\\\end{pmatrix}\in\mathbb{R}^q\otimes\mathbb{R}^p$.
Under our assumptions, we can deduce that
\begin{equation}
\hat{\Sig}_{n} \overset{P_0}\longrightarrow \Sig\quad\text{and}\quad
\hat{\Gam}_{n} \overset{P_0}\longrightarrow \Gam:=\begin{pmatrix}I_q&-B_0\\O&-\mathcal{I}(\theta_0)
\end{pmatrix}.
\nonumber
\end{equation}
Thus it follows from \eqref{hm:eq2} that we have the joint asymptotic normality of our estimators:
%$\Sigma^{-1/2}\Gamma\begin{pmatrix}\hat{u}_n\\ 
%\hat{v}_n\end{pmatrix}\overset{\mathcal{L}}\longrightarrow\mathcal{N}(0, I_{p+q})$, 
\begin{equation}
\begin{pmatrix}\hat{u}_n\\ 
\hat{v}_n\end{pmatrix}\overset{\mathcal{L}}\longrightarrow\mathcal{N}(0, \Gam^{-1}\Sig(\Gam^{-1})^T).
\nonumber
\end{equation}

\begin{Rem}\label{hm:rem_added1}
Recall that if $J_t$ is the standard Wiener process, then the rate of $\hat{\gamma}_{n}-\gamma_0$ is $\sqrt{n}$ (see \cite{Kessler1997}). 
The case $\int z^4\nu_0(dz)=0$ corresponds to this. 
As was noted in \cite{Masuda2013} (and as trivial from Lemma \ref{AN2} and Lemma \ref{AN3}), 
$\hat{\alpha}_{n}$ and $\hat{\gamma}_{n}$ are asymptotically orthogonal (hence asymptotically independent) if $\int z^{3}\nu_{0}(dz)=0$. 
Likewise, the asymptotic independence between $\hat{\theta}_n=(\hat{\al}_{n},\hat{\gam}_{n})$ and $(nh_{n})^{-1}\sum_{j}\vp(\hat{\delta}_j)$ 
can be easily seen from the expression of $\Sig$: in particular, 
$(\hat{\al}_n, \hat{\gam}_n)$ is asymptotically independent of $(nh_{n})^{-1}\sum_{j}\vp(\hat{\delta}_j )$ 
if both $\int \varphi_{k}(z)z \nu_0(dz)$ and $\int \varphi_{k}(z)z^2 \nu_0(dz)$ are zero.

%Then one may ask what about $\hat{\theta}_n=(\hat{\al}_{n},\hat{\gam}_{n})$ and $(nh_{n})^{-1}\sum_{j}\vp(\hat{\delta}_j)$. 
%This can be easily seen from the expression of $\Sig$: in particular, 
%$\hat{\al}_n$ and $\hat{\gam}_n$ cannot be asymptotically independent of $(nh_{n})^{-1}\sum_{j}\vp(\hat{\delta}_j )$ simultaneously, 
%for the quantities $\int \varphi_{k}(z)z \nu_0(dz)$ and $\int \varphi_{k}(z)z^2 \nu_0(dz)$ cannot equal zero at once.
\end{Rem}

%, for some function $f$. 
%It is valid in the limited condition: for any function $f:\mathbb{R}^p\times\mathbb{R}^q\to\mathbb{R}$ fulfilling $|f(x)|\lesssim1+|x|^{1-\epsilon}$, for some $\epsilon\in(0,1)$. This derives from the previous theorem and \cite[Theorem 2.20]{VanderVaart2000}.
%Combining Theorem \ref{MC1} and \cite[Theorem 2.20]{VanderVaart2000}, we obtain the next theorem.
%Define $\Gamma'\in\mathbb{R}^{p+q}\otimes\mathbb{R}^{p+q}$ by
%\begin{equation}
%\Gamma'=\begin{pmatrix}I_q&0\\
%0&\mathcal{I}^{-1}(\theta_0)\end{pmatrix}.
%\end{equation}
%For simplicity, we denote $(\sqrt{nh_n}\left\{\frac{1}{nh_n}\sum_{j=1}^n \varphi(\hat{\delta}_j)-\nu_0(\varphi)\right\}-\hat{b}_n\hat{w}_n,\hat{v}_n)$ by $\hat{\lambda}_n$.
%\begin{Thm}\label{MC2}
%Suppose that the condition of Theorem \ref{MC1} and Assumption \ref{Moment fitting function} hold.
%\begin{equation}
%E[f(\hat{\lambda}_n)]\longrightarrow\int_{\mathbb{R}^{p+q}}f(u)\phi(u;0,\Sigma')du,
%\end{equation}
%where $\Sigma'=\Gamma'\Sigma\Gamma'^T$ and for every continuous function $f:\mathbb{R}^p\to\mathbb{R}$ of at most polynomial growth.
%\end{Thm}

\medskip

Now we assume that the L\'{e}vy measure $\nu_{0}$ is parametrized by a parameter $\xi\in\Theta_\xi$, 
say $\nu_{\xi}$, where $\Theta_\xi$ is a bounded convex domain in $\mathbb{R}^q$, 
and that there exists a true value $\xi_0\in\Theta_\xi$. 
%Given a family of L\'{e}vy measures, we may find an appropriate $\varphi$ such that the equation 
%$F(\int\varphi(z)\nu_\xi(dz))=\xi$ has a solution $F\in\mathcal{C}^1$. 
The delta method then leads to the following corollary.

\begin{Cor}\label{Delta}
If the conditions of Corollary \ref{JAN2} hold and the equation $F(\int\varphi(z)\nu_\xi(dz), \theta)=(\xi,\theta)$ 
has a $\mathcal{C}^1$-solution $F: \mbbr^{q}\times\Theta\to\Theta_{\xi}\times\Theta$ such that $\p F(\nu_{\xi_0}(\varphi), \theta_{0})$ is invertible, then 
\begin{equation}
\p\hat{F}_n := \partial F\bigg(\frac{1}{nh_n}\sum_j\varphi(\hat{\delta}_j),\,
\hat{\theta}_n\bigg)\overset{P_{0}}\longrightarrow \p F(\nu_{\xi_0}(\varphi), \theta_{0}).
\nonumber
\end{equation}
Moreover, we have
\begin{equation}\nn
\{\p\hat{F}_n \hat{\Gamma}_n^{-1} \hat{\Sigma}_n^{1/2}\}^{-1}
\sqrt{nh_n}
(\hat{\xi}_n-\xi_0,\, 
\hat{\theta}_n-\theta_0)
\overset{\mathcal{L}}\longrightarrow\mathcal{N}(0,I_{p+q}),
\end{equation}
where $\hat{\xi}_n$ denotes the random vector consisting of the first $q$ elements of 
$F(\frac{1}{nh_n}\sum_j\varphi(\hat{\delta}_j), \hat{\theta}_{n})$.
\end{Cor}

\bigskip

\begin{Rem}
As a matter of fact, the absence of the Wiener part in the underlying SDE model \eqref{Model} 
is not essential in our results. Consider
\begin{equation}
dX_{t}=a(X_{t},\al)dt+\sig(X_{t},\gam)dW_{t}+c(X_{t-},\gam)dJ_{t},
\label{hm:rev.eq1}
\end{equation}
where $W$ is an $(\mcf_{t})$-adapted standard Wiener process independent of $(X_{0},J)$. 
We first note that the results of \cite{Masuda2013} still ensures the asymptotic normality of 
the corresponding GQMLE of $(\al,\gam)$ at rate $\sqrt{nh_{n}}$, in exchange for, in particular, 
some stringent identifiability condition on the scale parameter $\gam$; 
e.g. if $b(x,\gam)=\gam_{1}$ and $c(x,\gam)=\gam_{2}$ for $\gam=(\gam_{1},\gam_{2})$, 
then trivially we cannot estimate $\gam_{1}$ and $\gam_{2}$ separately by the naive Gaussian quasi-likelihood. 
Introducing an additional condition, we could deduce Theorem \ref{SE} 
with the same the expression \eqref{SE_b} of $\hat{b}_{n}$, 
except for the trivial change of the form of $\Sigma$, 
which stems from the necessary modification of the ``one-step'' variance 
in construction of the Gaussian quasi-likelihood corresponding to \eqref{hm:rev.eq1}, 
that is, from ``$hc_{t_{j-1}}^{2}(\gam)$'' to ``$h\{\sig_{t_{j-1}}^{2}(\gam)+c^{2}_{t_{j-1}}(\gam)\}$''; 
see Eq.(2.10) and the expression of $\Sig_{0}$ in \cite[pages 1600 and 1601]{Masuda2013} for details. 
More specifically, in the derivation of \eqref{SE}, 
which amounts to the stochastic expansions and estimates concerning 
the terms $b^{(1)}_{n}$ and $b^{(2)}_{n}$ in the proof of Lemma \ref{BCL}, 
it turned out that the presence of the Wiener part entails an additional condition 
on the behavior of the second derivative of $\vp$ around the origin, 
in order to make the remainder terms in the Taylor expansion indeed negligible. 
Also to be mentioned is that the independence between $W$ and $J$ is crucial in the computation of 
the leading-term of \eqref{SE}: formally, in applying \cite[Lemma 9]{Genon-catalot1993} we make use of the calculations such as 
$E^{j-1}_0[g(\D_j J)\D_j W]=0$ and $E^{j-1}_0[g^{2}(\D_j J)(\D_j W)^2]=h_nE[g^{2}(\D_j J)]$ to obtain $\frac{1}{\sqrt{nh_n}}\sumj f_{t_{j-1}}g(\D_j J)\D_j W=o_p(1)$ for suitable $f$ and $g$ with $g(0)=0$. 
%
%Also, some care is necessary when following the moment estimates in the present proofs; 
%for example, the moment estimate \eqref{hm:add.eq6} of $\D_{j}J-\del_{j}$ necessary in the proof get changed: 
%for any $q\geq2$, the $q$th-absolute moment of 
%${\del}_{j}-\D_{j}J\overset{P_0}= {c}_{t_{j-1}}^{-1}\{\int_{j}(a_{s}-a_{t_{j-1}})ds 
%+ \int_{j}(b_{s}-b_{t_{j-1}})dW_{s} + \int_{j}(c_{s-}-c_{t_{j-1}})dJ_{s} +b_{t_{j-1}}\D_j W\}$ 
%is of order $h_n^{(q/2)\wedge2}$ 
%(the leading-term is $b_{t_{j-1}}\D_j W$ if $2\leq q\leq4$, otherwise $\int_{j}(c_{s-}-c_{t_{j-1}})dJ_{s}$). 
%
Building on these observations, the proofs in case of \eqref{hm:rev.eq1} go through as in \eqref{Model} without further difficulty, 
while the precise statement concerning the model \eqref{hm:rev.eq1} requires a series of changes of notation. 
We would like to omit details of the full picture.
\label{hm:rem_wiener.part}
\end{Rem}

\color{black}

\section{Numerical experiments}\label{simulations}
Consider the following one-dimensional L\'{e}vy driven SDE:
\begin{equation}\label{cModel}
dX_t=-\alpha X_tdt - \frac{\gamma}{1+X_{t-}^2}dJ_t,\quad X_0=0,
\end{equation}
where the true value is $(\alpha_0,\gamma_0)=(0.5,0.2)$; the driving noise process is the normal inverse Gaussian L\'{e}vy process such that $\mathcal{L}(J_t)=NIG(\delta,0,\delta t,0)$ with $\delta=1,5$, or $10$.
It is well known that the cumulant function of $J_1$ is explicitly given by
\begin{equation}\label{cum}
\kappa(u):=\log E[\exp(iuJ_1)]=\int(\cos(uz)-1)\nu_{0,\del}(dz)=\delta(\delta-\sqrt{\delta^2+u^2}).
\end{equation}
where $\nu_{0,\delta}$ denotes the L\'{e}vy measure of $NIG(\delta,0,\delta,0)$. 
Hence it follows that $E[J_1]=0, E[J_1^2]=1,$ and that $\mathcal{L}(J_t)\overset{\mathcal{L}}\longrightarrow N(0,t)$ as $\delta\to\infty$.

In addition to $(\alpha_0,\gamma_0)$, we estimate the value of $\kappa(u)$ for some $u$, so
by the symmetry of $\mathcal{L}(J_t)$ we set the moment fitting function $\varphi(x,u)=\cos(ux)-1$. 
Note that the SDE model \eqref{cModel} satisfies all of our assumptions; see \cite[Proposition 5.4]{Masuda2013} for the stability condition. 

Put $T_{n}=nh_{n}$. 
Our simulations were done for $(T_n,h_n)=(10,0.05),(50,0.025)$ and $(100,0.01)$ with respect to each $\delta$. 
We simulated 1000 independent sample paths for true model with sufficiently small step size by use of Euler scheme and the 1000 estimates $(\hat{\alpha}_n,\hat{\gamma}_n,\widehat{\kappa(1)}_n,\widehat{\kappa(3)}_n,\widehat{\kappa(5)}_n)$, where
\begin{equation}
\widehat{\kappa(u)}_n:=\frac{1}{nh_n}\sum_{j=1}^n\varphi(\hat{\delta}_j,u),
\nonumber
\end{equation}
were calculated for each sample path. 
For generating sample paths, we used {\tt yuima} package \cite{yuima} for R statistical environment \cite{rct}. 
The mean and the standard deviation of $(\hat{\alpha}_n,\hat{\gamma}_n,\widehat{\kappa(1)}_n,\widehat{\kappa(3)}_n,\widehat{\kappa(5)}_n)$ were computed; these are shown in Table \ref{1}-\ref{5}.

%\tcb{
%By means of Theorem \ref{Bias correction}, the explicit form of the bias correction term is given by:
%\begin{equation*}
%\hat{b}_n(u)=\frac{\hat{\gamma}_n^{-1}}{nh_n}\sum_{j=1}^n u\hat{\delta}_j\sin(u\hat{\delta}_j),\quad \mbox{for} \ u=1,3,5.
%\end{equation*}
%To visualize the effect of bias correction term, we give QQ-plot of 
%\begin{equation}
%C_{n}(u) := \sqrt{nh_n}\big\{\widehat{\kappa(u)}_n-\kappa(u)\big\}-\hat{b}_n(u)\hat{w}_n
%\nonumber
%\end{equation}
%%$\sqrt{nh_n}(\widehat{\kappa(u)}_n-\kappa(u))-\hat{b}_n(u)\hat{w}_n$ 
%for each $u$ with $(T_n,h_n)=(100,0.01)$ and $\del=10$ in Figures \ref{kappa1}--\ref{kappa5}.
%}

From the results, we can observe the following:
\begin{itemize}
\item the performance of $\hat{\alpha}_n$ can be affected not by the value of $\delta$ but by the value of $T_n$;
\item the performance of $\hat{\gamma}_n$ seems to improve in terms of standard deviation as the value of $\delta$ increases, 
which can be thought to come from the fact that the asymptotic variance of $\hat{\gamma}_n$ tends to 0 as $\delta\to\infty$ (we have $\int z^4 \nu_0(dz)=3\delta^{-2}$);
\item the performance of $\widehat{\kappa(u)}_n$ becomes better for smaller $u$. 
This is quite natural because 
by Theorem \ref{JAN1} the asymptotic variance of $\widehat{\kappa(u)}_n$ is $\int \varphi(x,u)^2\nu_{0,\delta}(dx)$. 
From the half-angle formula $\varphi(x,u)^2=(\cos(ux)-1)^2=-2\varphi(x,u)+\frac{1}{2}\varphi(x,2u)$, 
we have
\begin{align*}
\int \varphi(x,u)^2\nu_{0,\delta}(dx)&=\int\left(-2\varphi(x,u)+\frac{1}{2}\varphi(x,2u)\right) \nu_{0,\delta}(dx)\\
&=-\frac{3}{2}\delta^2+\delta\left(2\sqrt{\delta^2+u^2}-\frac{1}{2}\sqrt{\delta^2+4u^2}\right)=:f(\delta,u).
\end{align*}
Since $\partial_u f(\delta,u)=2\delta(u/\sqrt{\delta^2+u^2}-u/\sqrt{\delta^2+4u^2})>0$ for all $\delta>0$, 
the asymptotic variance of $\widehat{\kappa(u)}_n$ is increasing in $u$, 
clarifying better performance of $\widehat{\kappa(q)}_n$ for smaller value of $q$. 
\end{itemize}

In this example, it should be noted that a large value of $u$ brings about large finite-sample bias and variance of 
the scaled estimators: $\sqrt{nh_{n}}\{\widehat{\kappa(u)}_n - \kappa(u)\}$, 
because then 
both the term $\hat{b}_{n}=\hat{b}_{n}(u)$ and the $o_{p}(1)$ term in the right-hand side of \eqref{SE} 
will become large in an increasing way with the value $|u|$; as seen from the proof, 
the latter term involves higher-order partial derivatives of $\vp(x,u)$ with respect to $x$.

\begin{table}[t]
\begin{center}
\caption{The performance of the two-step type estimators with $\delta=1$ and the true value $(\alpha_0,\gamma_0,\kappa(1),\kappa(3),\kappa(5))=(0.5,0.2,-0.4142,-2.1623,-4.0990)$; the mean is given with the standard deviation in parentheses.}
\begin{tabular}{ccccccc}
\hline
&&&&&&\\[-3.5mm]
$T_n$ & $h_n$ & $\hat{\alpha}_n$ & $\hat{\gamma}_n$ & $\widehat{\kappa(1)}_n$ & $\widehat{\kappa(3)}_n$ & $\widehat{\kappa(5)}_n$\\[1mm] \hline
10&0.05&0.6609&0.1923&-0.4455&-2.2888&-4.0765 \\
&&(0.3912)&(0.0489)&(0.0340)&(0.5594)&(0.9300)\\ 
50&0.025&0.5394&0.1967&-0.4243&-2.1713&-4.0020\\  
&&(0.1396)&(0.0224)&(0.0346)&(0.2801)&(0.4929)\\ 
100&0.01&0.5205&0.1986&-0.4198&-2.1730&-4.0755\\ 
&&(0.0994)&(0.0163)&(0.0292)&(0.2087)&(0.3746)\\[1mm] \hline
\end{tabular}
\label{1}
\end{center}
\end{table}

\begin{table}[t]
\begin{center}
\caption{The performance of the two-step type estimators with $\delta=5$ and the true value $(\alpha_0,\gamma_0,\kappa(1),\kappa(3),\kappa(5))=(0.5,0.2,-0.4951,-4.1548,-10.3553)$; the mean is given with the standard deviation in parentheses.}
\begin{tabular}{ccccccc}
\hline
&&&&&&\\[-3.5mm]
$T_n$ & $h_n$ & $\hat{\alpha}_n$ & $\hat{\gamma}_n$ & $\widehat{\kappa(1)}_n$ & $\widehat{\kappa(3)}_n$ & $\widehat{\kappa(5)}_n$\\[1mm] \hline
10&0.05&0.6762&0.1969&-0.4898&-3.7868&-8.1786\\
&&(0.4052)&(0.0144)&(0.0029)&(0.1469)&(0.5172)\\
50&0.025&0.5302&0.1989&-0.4921&-3.9482&-9.1297\\
&&(0.1421)&(0.0056)&(0.0016)&(0.0725)&(0.2480)\\
100&0.01&0.5160&0.1995&-0.4939&-4.0726&-9.8473\\
&&(0.1000)&(0.0040)&(0.0010)&(0.0479)&(0.1758)\\[1mm] \hline
\end{tabular}
\label{3}
\end{center}
\end{table}

\begin{table}[t]
\begin{center}
\caption{The performance of the two-step type estimators with $\delta=10$ and the true value $(\alpha_0,\gamma_0,\kappa(1),\kappa(3),\kappa(5))=(0.5,0.2,-0.4988,-4.4031,-11.8034)$; the mean is given with the standard deviation in parentheses.}
\begin{tabular}{ccccccc}
\hline
&&&&&&\\[-3.5mm]
$T_n$ & $h_n$ & $\hat{\alpha}_n$ & $\hat{\gamma}_n$ & $\widehat{\kappa(1)}_n$ & $\widehat{\kappa(3)}_n$ & $\widehat{\kappa(5)}_n$\\[1mm] \hline
10&0.05&0.6785&0.1962&-0.4928&-3.9645&-8.9678\\
&&(0.4130)&(0.0109)&(0.0013)&(0.0762)&(0.3412)\\
50&0.025&0.5250&0.1985&-0.4957&-4.1710&-10.2291\\
&&(0.1391)&(0.0039)&(0.0004)&(0.0234)&(0.1184)\\
100&0.01&0.5160&0.1994&-0.4975&-4.3084&	-11.1378\\
&&(0.0972)&(0.0023)&(0.0002)&(0.0116)&(0.0649)\\[1mm] \hline
\end{tabular}
\label{5}
\end{center}
\end{table}

%\begin{figure}[h]
%\begin{center}
%\begin{tabular}{c}
% \begin{minipage}{0.33\hsize}
%  \begin{center}
%   \includegraphics[width=50mm]{kappa1.eps}
%  \end{center}
%  \caption{QQ-plot of $C_{n}(1)$.}
%  \label{kappa1}
% \end{minipage}
% \begin{minipage}{0.33\hsize}
% \begin{center}
%  \includegraphics[width=50mm]{kappa3.eps}
% \end{center}
%  \caption{QQ-plot of $C_{n}(3)$.}
%  \label{kappa3}
% \end{minipage}
% \begin{minipage}{0.33\hsize}
% \begin{center}
%  \includegraphics[width=50mm]{kappa5.eps}
% \end{center}
%  \caption{QQ-plot $C_{n}(5)$.}
%  \label{kappa5}
% \end{minipage}
% \end{tabular}
% \end{center}
%\end{figure}

%%%%%
%%%%%

\section{Proofs}\label{Appendix}

Throughout our proofs, we will often omit ``$n$'' of the notation $h_n$ and write $E$ instead of $E_0$. 

\subsection{Preliminary lemmas}

We begin with some lemmas.

\begin{Lem}\label{MC}
Suppose that Assumption \ref{Sampling design} and Assumption \ref{Moments} hold.
For all $q\geq2$, it follows that
\begin{equation}\nn
\frac{1}{h}E[|J_h|^q]\to\int|z|^q\nu_0(dz).
\end{equation}
\begin{proof}
Under Assumption \ref{Moments}, $\varphi(z)=|z|^q$ satisfies the condition of \cite[Theorem 1]{Lopez2008}. 
\end{proof}
\end{Lem}

\begin{Rem}
Although the above convergence might not be valid for all $0<q<2$, it holds when $q\geq\beta$, where $\beta$ denotes the Blumenthal-Getoor index (for details, see \cite[Theorem 1]{Lopez2008}, \cite[Section 5.2]{Jac07}, and \cite[Theorem 1]{LusPag08}).
\end{Rem}

From now on we simply write $f_{j-1}(\theta)=f(X_{t_{j-1}},\theta)$, $f_{j-1}=f(X_{t_{j-1}},\theta_{0})$ 
and $\hat{f}_{j-1}=f(X_{t_{j-1}},\hat{\theta}_{n})$.

\begin{Lem}\label{MF}
Let $f:\mathbb{R}\times\Theta_\alpha\times\Theta_\gamma\mapsto\mathbb{R}$ be a polynomial growth function with respect to $x$, uniformly in $\alpha$ and $\gamma$. 
If Assumptions \ref{Sampling design}-\ref{Stability} are satisfied, then, 
for all $p\in\left\{1,2\right\}$ and $q\geq0$ it follows that
\begin{equation}\nn
\sup_n \sup_\theta E\left[\left|\frac{1}{nh}\sum_jf_{j-1}(\theta)(\Delta_jX -ha_{j-1}(\alpha))^p\right|^q\right]<\infty.
\end{equation}

Moreover, we have 
\begin{align*}
&\sup_nE\left[\left|\frac{1}{\sqrt{nh}}\sum_jf_{j-1}(\Delta_jX-ha_{j-1})\right|^q\right]<\infty,\\
&\sup_nE\left[\left|\frac{1}{\sqrt{nh}}\sum_j\left\{f_{j-1}(\Delta_jX-ha_{j-1})^2-hf_{j-1} c^2_{j-1}\right\}\right|^q\right]<\infty.
\end{align*}
\end{Lem}

\begin{proof}
First, we show the case of $p=1$ and $q\geq 2$.
By the definition of $X$, we have
\begin{align*}
&E\left[\left|\frac{1}{nh}\sum_jf_{j-1}(\theta)(\Delta_jX -ha_{j-1}(\alpha))\right|^q\right]\\
&\lesssim E\left[\left|\frac{1}{nh}\sum_jf_{j-1}(\theta) \int_j(a_s-E^{j-1}[a_s])ds\right|^q\right] +E\left[\left|\frac{1}{nh}\sum_jf_{j-1}(\theta) \int_j(E^{j-1}[a_s]-a_{j-1})ds\right|^q\right] \\
&+E\left[\left|\frac{1}{n}\sum_jf_{j-1}(\theta)(a_{j-1}-a_{j-1}(\alpha))\right|^q\right]+ E\left[\left|\frac{1}{nh}\sum_jf_{j-1}(\theta) \int_j c_{s-}dJ_s \right|^q\right].\\
\end{align*}
We will check separately that all terms are finite. 
From the assumption on $f$ and Jensen's inequality, we get
\begin{equation}\nn
E\left[\left|\frac{1}{n}\sum_jf_{j-1}(\theta)(a_{j-1}-a_{j-1}(\alpha))\right|^q\right] \leq \frac{1}{n}\sum_jE\left[\left|f_{j-1}(\theta)(a_{j-1}-a_{j-1}(\alpha))\right|^q\right]<\infty.
\end{equation}
By It\^{o}'s formula, we have 
\begin{equation}\nn
E^{j-1}[a_s]-a_{j-1}=\int_{t_{j-1}}^s E^{j-1}[\tilde{\mathcal{A}} a_u] du,
\end{equation}
where $\tilde{\mathcal{A}}$ denotes the formal infinitesimal generator of $X$, namely, for $f\in\mathcal{C}^1(\mathbb{R})$,
\begin{equation}\nn
\tilde{\mathcal{A}}f(x)=\partial f(x) a(x)+ \int (f(x+c(x)z)-f(x)-\partial f(x)c(x)z) \nu_0(dz).
\end{equation} 
By \cite[Lemma 4.5]{Masuda2013}, the definition of $\tilde{\mathcal{A}}$ 
and the assumptions about coefficients and moments, for a $v\in(0,1)$, we get
\begin{align*}
\left|E^{j-1}[\tilde{\mathcal{A}} a_u]\right| &\leq E^{j-1}\left[\left|(\partial_x a_u) a_u+ \int (a(X_u+c_u z)-a_u-(\partial_x a_u) c_u z) \nu_0(dz) \right|\right]\\
&\lesssim E^{j-1}\left[1+|X_u|^C+\int |\partial^2_x a(X_u+v c_u z) (c_uz)^2| \nu_0(dz)\right]\\
&\lesssim E^{j-1}\left[1+|X_u|^C\right] \lesssim E^{j-1}\left[1+|X_u-X_{j-1}|^C+|X_{j-1}|^C\right] \lesssim 1+|X_{j-1}|^C.
\end{align*}
Note that we used the fact that $\int z^q \nu_0(dz)<\infty$ for any $q\geq2$. 
Hence it follows that
\begin{align*}
E\left[\left|\frac{1}{nh}\sum_jf_{j-1}(\theta) \int_j(E^{j-1}[a_s]-a_{j-1})ds\right|^q\right] &\leq E\left[\left|\frac{1}{nh}\sum_j|f_{j-1}(\theta)| \left|\int_j(E^{j-1}[a_s]-a_{j-1})ds\right|\right|^q\right]\\ 
& 
\leq E\left[\left|\frac{1}{nh}\sum_j|f_{j-1}(\theta)| \int_j\int_{t_{j-1}}^s
\left| E^{j-1}[\tilde{\mathcal{A}} a_u] \right|duds\right|^q\right]
\\
&\lesssim E\left[\left|\frac{h}{n}\sum_j|f_{j-1}(\theta)|(1+|X_{j-1}|^C)\right|^q\right]  \lesssim  h^q <\infty.
\end{align*}
Burkholder's inequality for martingale difference array yields that
\begin{align*}
E\left[\left|\frac{1}{nh}\sum_jf_{j-1}(\theta) \int_j(a_s-E^{j-1}[a_s])ds\right|^q\right] &\lesssim n^{-\frac{q}{2}-1} \sum_jE\left[\left|f_{j-1}(\theta) \frac{\int_j(a_s-E^{j-1}[a_s])ds}{h}\right|^q\right]\\
&\lesssim n^{-\frac{q}{2}-1} \sum_j\sqrt{E\left[\left|\frac{\int_j(a_s-E^{j-1}[a_s])ds}{h}\right|^{2q}\right]}\\
&\leq n^{-\frac{q}{2}-1}\sum_j\sqrt{\frac{1}{h} E\left[\int_j |a_s-E^{j-1}[a_s]|^{2q} ds\right]}\\
%&\lesssim n^{-\frac{q}{2}-1}\sum_j\sqrt{\frac{1}{h} E\left[\int_j |a_s-a_{j-1}+a_{j-1}-E^{j-1}[a_s]|^{2q} ds\right]}\\
&\lesssim n^{-\frac{q}{2}-1}\sum_j\sqrt{\frac{1}{h} \int_j E\left[|a_s-a_{j-1}|^{2q}+|a_{j-1}-E^{j-1}[a_s]|^{2q}\right] ds}\\
&\lesssim n^{-\frac{q}{2}} \sqrt{h}<\infty.
\end{align*}
Define the indicator function $\chi_j$ by
\begin{equation}\nn
\chi_j(s)=\begin{cases}
1&s\in(t_{j-1},t_j],\\
0&\text{otherwise}. \end{cases}
\end{equation}
Using this indicator function and Burkholder's inequality, we can obtain
\begin{align}
E\left[\left|\frac{1}{nh}\sum_jf_{j-1}(\theta) \int_j c_{s-}dJ_s \right|^q\right] &=E\left[\left|\frac{1}{nh}\int^{nh}_0\sum_jf_{j-1}(\theta) c_{s-}\chi_j(s)dJ_s \right|^q\right] \nn \\
&\lesssim(nh)^{-\frac{q}{2}-1} \int^{nh}_0 E\left[\left|\sum_jf_{j-1}(\theta) c_{s}\chi_j(s)\right|^q\right] ds\nn\\
&=(nh)^{-\frac{q}{2}-1} \sum_j\int_j E\left[\left| f_{j-1}(\theta) c_{s}\right|^q\right] ds \nn\\
&\lesssim(nh)^{-\frac{q}{2}}<\infty.
\label{hm.eq1}
\end{align}
Second, we look at the cases of $p=2$ and $q\geq 2$. Quite similarly to the above, we have 
\begin{align*}
&E\left[\left|\frac{1}{nh}\sum_jf_{j-1}(\theta)(\Delta_jX -ha_{j-1}(\alpha))^2\right|^q\right]\\
&\lesssim E\left[\left|\frac{1}{nh}\sum_jf_{j-1}(\theta) \left(\int_j(a_s-E^{j-1}[a_s])ds\right)^2\right|^q\right] +E\left[\left|\frac{1}{nh}\sum_jf_{j-1}(\theta) \left(\int_j(E^{j-1}[a_s]-a_{j-1})ds\right)^2\right|^q\right] \\
&+E\left[\left|\frac{1}{n}\sum_jf_{j-1}(\theta)(a_{j-1}-a_{j-1}(\alpha))^2\right|^q\right]+ E\left[\left|\frac{1}{nh}\sum_jf_{j-1}(\theta)\left( \int_j c_{s-}dJ_s \right)^2 \right|^q\right].\\
\end{align*}
In the same way, we get
\begin{align*}
&E\left[\left|\frac{1}{nh}\sum_jf_{j-1}(\theta) \left(\int_j(E^{j-1}[a_s]-a_{j-1})ds\right)^2\right|^q\right] \lesssim h^{3q} <\infty,\\
&E\left[\left|\frac{1}{n}\sum_jf_{j-1}(\theta)(a_{j-1}-a_{j-1}(\alpha))^2\right|^q\right] <\infty.
\end{align*}
Jensen's inequality implies that
\begin{align*}
E\left[\left|\frac{1}{nh}\sum_jf_{j-1}(\theta) \left(\int_j(a_s-E^{j-1}[a_s])ds\right)^2\right|^q\right] 
&\leq E\left[\left|\frac{h}{n}\sum_j|f_{j-1}(\theta)| \left(\frac{\int_j(a_s-E^{j-1}[a_s])ds}{h}\right)^2\right|^q\right]\\
&\leq E\left[\left|\frac{1}{n}\sum_j|f_{j-1}(\theta)| \int_j|a_s-E^{j-1}[a_s]|^2ds\right|^q\right]\\
&=E\left[\left|\frac{1}{n}\int^{nh}_0 \sum_j|f_{j-1}(\theta)| |a_s-E^{j-1}[a_s]|^2 \chi_j(s) ds\right|^q\right]\\
&\leq h^q \frac{1}{nh} E \left[\int^{nh}_0 \left|\sum_j|f_{j-1}(\theta)| |a_s-E^{j-1}[a_s]|^2 \chi_j(s)\right|^q ds\right]\\
&=\frac{h^{q-1}}{n} \sum_jE\left[\int_j |f_{j-1}(\theta)|^q |a_s-E^{j-1}[a_s]|^{2q}ds\right]\\
&\lesssim h^q \sqrt{h} < \infty.
\end{align*}
From It\^{o}'s formula, we get 
\begin{align*}
\left( \int_j c_{s-}dJ_s \right)^2&=2 \int_j \left(\int_{t_{j-1}}^s c_{u-} dJ_u\right) c_{s-} dJ_s+\int_j \int c_{s-}^2 z^2 N(ds,dz)\\
&=2 \int_j \left(\int_{t_{j-1}}^s c_{u-} dJ_u\right) c_{s-} dJ_s+\int_j \int c_{s-}^2 z^2 \tilde{N}(ds,dz)+\int_j c_{s}^2 ds\\
&=2 \int_j \left(\int_{t_{j-1}}^s c_{u-} dJ_u\right) c_{s-} dJ_s+\int_j \int c_{s-}^2 z^2 \tilde{N}(ds,dz)+\int_j (c_{s}^2-E^{j-1}\left[c_{s}^2\right])ds\\
&+\int_j (E^{j-1}\left[c_{s}^2\right]-c^2_{j-1})ds+h c^2_{j-1},
\end{align*}
where $N(ds,dz)$ (resp. $\tilde{N}(ds,dz)$) is the Poisson random measure 
(resp. compensated Poisson random measure) associated with $J$.
It follows from this decomposition together with a similar estimate to \eqref{hm.eq1} that
\begin{equation}\nn
E\left[\left|\frac{1}{nh}\sum_jf_{j-1}(\theta)\left( \int_j c_{s-}dJ_s \right)^2 \right|^q\right]<\infty.
\end{equation}
If $\theta=\theta_0$, we do not have to consider the term containing $a_{j-1}-a_{j-1}(\alpha)$.
Hence Jensen's inequality gives the desired result for all $q\geq0$.
\end{proof}

For the sake of the asymptotic normality of $u_n$, we introduce the function space:
\begin{align*}
\mathcal{K}_1
&=\Biggl\{f=(f_{k}):\mbbr\to\mbbr^{q}\, \Bigg|\, \text{$f$ is of class $C^{2}$}, \quad \nu_0(f)<\infty, \quad
\sup_{0\le s\le h_{n}}E_0\left[|\mathcal{A}^2 f(J_s)|\right]=O(1),
\nonumber \\
&{}\qquad \text{and}\quad
\max_{i=0,1}\int_0^{h_n}\int E_0\left[|\mathcal{A}^i f(J_{s-}+z)-\mathcal{A}^i f(J_{s-})|^2\right]\nu_0(dz) ds=O(1)\Biggr\}.
\nonumber\\
\end{align*}

\begin{Lem}\label{Rate of gap}
If Assumption \ref{Sampling design} holds and if the function $f:\mbbr\to\mbbr^q$ fulfills that 
$f(0)=\p f(0)=0$ and $f\in\mathcal{K}_{1}$, then we have
\begin{equation}\nn
\frac{1}{h}E[f(J_h)]-\nu_0(f)=O(h).
\end{equation}
\end{Lem}

\begin{proof}
By It\^{o}-Taylor expansion, we see that
\begin{align*}
f(J_h)&=f(0)+h\mathcal{A}f(0)+\int_0^h\int_0^s\mathcal{A}^2 f(J_u)duds\\
&+\int_0^h\int\left\{f(J_{s-}+z)-f(J_{s-})\right\}\tilde{N}(ds,dz) 
+\int_0^h \int_0^s\int\left\{\mathcal{A}f(J_{u-}+z)-\mathcal{A}f(J_{u-})\right\}\tilde{N}(du,dz)ds,
\end{align*}
Under the assumptions we see that the last two terms are martingale (see \cite[Theorem 4.2.3]{Applebaum2009}) 
and $\mathcal{A}f(0)=\nu_0(f)$, hence the result follows. 
\end{proof}

\begin{Lem}\label{AN}
Suppose that Assumption \ref{Sampling design}, Assumption \ref{Moments} and Assumption \ref{Moment fitting function} hold.
Then we have
\begin{equation}\nn
u_n \overset{\mathcal{L}} \longrightarrow \mathcal{N}_q (0,\Sigma_{11}),
\end{equation}
\end{Lem}

\begin{proof}
By the stationarity and independence of increments of L\'{e}vy process $J$, we have
\begin{align*}
u_n &=\sqrt{nh}\left\{\frac{1}{nh}\sum_j\left(\varphi(\Delta_jJ) -E^{j-1}[\varphi(\Delta_jJ)]\right)\right\}+  \sqrt{nh}\left\{\frac{1}{nh}\sum_jE^{j-1}[\varphi(\Delta_jJ)]-\nu_0(\varphi)\right\}\\
&=: e_n+f_n,
\end{align*}
where
\begin{align*}
e_n &=\sqrt{nh}\left\{\frac{1}{nh}\sum_j\left(\varphi(\Delta_jJ) -E[\varphi(J_h)]\right)\right\}=: \frac{1}{\sqrt{nh}} \sum_j e_{j},\\
f_n &=\sqrt{nh} \left(\frac{1}{h} E[\varphi(J_h)]-\nu_0(\varphi)\right).
\end{align*}
By the previous lemma, it is clear that $f_n=o(1)$ under $nh_{n}^{2}\to 0$.
We will prove that $e_n \overset{\mathcal{L}} \longrightarrow \mathcal{N}_q (0,\Sigma_{11})$ 
by applying the martingale central limit theorem \cite{Dvoretzky1972}. 
First, we show that $\varphi_k, \varphi_k\varphi_l \in\mathcal{K}_1$ where $\varphi_k$ denotes $k$th component of $\varphi$ (in the case of $q=1$).
We only prove $\varphi\in\mathcal{K}_1$; the other case is similar.
By the definition of $\mathcal{A}$ (see \eqref{IG}) and Taylor's theorem, for fixed $s\in(0,1)$, we have
\begin{equation}\nn
E\left[\left|\mathcal{A}^2 \varphi (J_s) \right|\right] \lesssim \sup_{u\in[0,1]} E\left[ \left|\int \partial^2 \mathcal{A} \varphi (J_s+uz)z^2 \nu_0(dz)\right| \right].
\end{equation}
Recall that by Assumption \ref{Moments} and the definition of L\'{e}vy measure, we have $\int |z|^q\nu_0(dz)<\infty$ for all $q\geq2$.
By means of  Assumption \ref{Moment fitting function} and dominated convergence theorem it follows that
\begin{align*}
|\partial^2 \mathcal{A} \varphi(x)| &=\left|\partial^2 \left(\int (\varphi(x+z)-\varphi(x)-\partial\varphi(x)z) \nu_0(dz)\right)\right| \lesssim 1+|x|^C,
\end{align*} 
for all $x\in\mathbb{R}$.
Hence we have 
\begin{equation*}
\sup_{0\leq s \leq h_n}E\left[|\mathcal{A}^2\varphi(J_s)|\right]<\infty.
%\label{hm:add.eq1}
\end{equation*}
Similarly, we can show that 
\begin{align*}
&\int_0^{h_n} \int E\left[|\varphi(J_{s-}+z)-\varphi(J_{s-})|^2\right] \nu_0(dz) ds<\infty,\\
%\label{hm:add.eq2}
&\int_0^{h_n} \int E\left[|\mathcal{A}\varphi(J_{s-}+z)-\mathcal{A}\varphi(J_{s-})|^2\right] \nu_0(dz) ds<\infty.
%\label{hm:add.eq3}
\end{align*}
Obviously, we have $E[e_j] =0$.
The properties of conditional expectation yield that
\begin{equation}\nn
E[e_{j_k}e_{j_l}]=E[\varphi_k(J_h)\varphi_l(J_h)]-E[\varphi_k(J_h)]E[\varphi_l(J_h)],
\end{equation}
Lemma \ref{Rate of gap} leads to $\frac{1}{nh} \sum_jE[e_{j_k}e_{j_l}] \longrightarrow \int \varphi_k(z)\varphi_l(z) \nu_0(dz)$.
From Assumption \ref{Moments}, Assumption \ref{Moment fitting function} and Lemma \ref{MC}, we obtain $E[|e_{j_k}|^4] = O(h)$.
Hence we have $\frac{1}{(nh)^2}\sum_jE[|e_j|^4]\to0$, namely Lindeberg condition holds.
Combining these discussion, we deduce that 
$e_n \overset{\mathcal{L}}\longrightarrow \mathcal{N}_q(0,\Sigma_{11})$ as was to be shown, completing the proof.
\end{proof}

Define $\tilde{G}_n(\theta)\in\mathbb{R}^p$ by 
\begin{equation}\nn
\tilde{G}_n(\theta)=(G_n^\alpha(\theta),G_n^\gamma(\theta)),
\end{equation}
and it is easy to calculate its derivative $\partial_\theta \tilde{G}_n(\theta)=\begin{pmatrix} \partial_\alpha G_n^\alpha(\theta)&\partial_\gamma G_n^\alpha(\theta)\\ \partial_\alpha G_n^\gamma(\theta)&\partial_\gamma G_n^\gamma(\theta) \end{pmatrix}\in\mathbb{R}_p\otimes\mathbb{R}_p$ as follows:
\begin{align*}
&\partial_\alpha G_n^\alpha(\theta)=\frac{1}{nh}\sum_j\left\{\frac{\partial_\alpha^{\otimes 2}a_{j-1}(\alpha)}{c^2_{j-1}(\gamma)}(\Delta_jX-h a_{j-1}(\alpha))-h\frac{(\partial_\alpha a_{j-1}(\alpha))^{\otimes 2}}{c^2_{j-1}(\gamma)}\right\},\\
&\partial_\gamma G_n^\alpha(\theta)=\frac{1}{nh}\sum_j\partial_\alpha a_{j-1}[\partial_{\gamma^T}c_{j-1}^{-2}(\gamma)](\Delta_jX-h a_{j-1}(\alpha)),\\
&\partial_\alpha G_n^\gamma(\theta)={\frac{2}{n}}\sum_j(\Delta_jX-h a_{j-1}(\alpha))[\partial_\gamma c_{j-1}^{-2}(\gamma)] \partial_{\alpha^T} a_{j-1}(\alpha), \\
&\partial_\gamma G_n^\gamma(\theta)=\frac{1}{nh}\sum_j\left\{[-\partial_\gamma^{\otimes 2}c_{j-1}^{-2}(\gamma)] (\Delta_jX-ha_{j-1}(\alpha))^2+2h\frac{(\partial_\gamma c_{j-1}(\gamma))^{\otimes 2}-c_{j-1}(\gamma)\partial_\gamma^{\otimes 2}c_{j-1}(\gamma)}{c_{j-1}^2(\gamma)}\right\}.
\end{align*}.

\begin{Lem}\label{AN2}
Under Assumptions \ref{Sampling design}-\ref{Identifiability}, it follows that
\begin{align*}
&\sup_{\theta\in\Theta}\left|\tilde{G}_n(\theta)-G_\infty(\theta)\right|\overset{P_0}\longrightarrow0,\\
&\sqrt{nh}\tilde{G}_n(\theta_0)\overset{\mathcal{L}}\longrightarrow\mathcal{N}_p(0,\Sigma_{22}).
\end{align*}
\end{Lem}

\begin{proof}
For simplicity, we do suppose that $p_\alpha=p_\gamma=1$; the high dimensional case is similar.  
First, we will show the $\theta$-pointwise convergence
\begin{equation}\nn 
\left|\tilde{G}_n(\theta)-G_\infty(\theta)\right|\overset{P_0}\longrightarrow0.
\end{equation}
From \cite[Lemma 9]{Genon-catalot1993}, it suffices to that show 
\begin{align*}
&\frac{1}{nh} \sum_jE^{j-1}\left[M_{j-1}(\theta) (\Delta_jX -h a_{j-1}(\alpha))\right] \overset{P_0} \longrightarrow G_\infty^{\alpha}(\theta),\\
&\frac{1}{nh}  \sum_jE^{j-1}\left[\left[-\partial_\gamma c_{j-1}^{-2}(\gamma)\right](\Delta_jX-h a_{j-1}(\alpha))^2-h \frac{\partial_\gamma c_{j-1}^2(\gamma)}{c_{j-1}^2(\gamma)}\right]\overset{P_0} \longrightarrow G_\infty^{\gamma}(\gamma),\\
&\frac{1}{(nh)^2}  \sum_jE^{j-1}\left[\left|M_{j-1}(\theta) (\Delta_jX -h a_{j-1}(\alpha))\right|^2\right] \overset{P_0} \longrightarrow 0,\\
&\frac{1}{(nh)^2} \sum_jE^{j-1}\left[\left|\left[-\partial_\gamma c_{j-1}^{-2}(\gamma)\right](\Delta_jX-h a_{j-1}(\alpha))^2-h \frac{\partial_\gamma c_{j-1}^2(\gamma)}{c_{j-1}^2(\gamma)}\right|^2\right]\overset{P_0} \longrightarrow 0.
\end{align*}
By the definition of $X$, we observe that
\begin{equation}\nn
\Delta_jX-h a_{j-1}(\alpha)=\int_j (a_s-a_{j-1})ds+\int_j c_{s-} dJ_s+h (a_{j-1}-a_{j-1}(\alpha)).
\end{equation}
Hence the martingale property of $\int_{t_{j-1}}^s c_{u-} dJ_u$ implies that
\begin{equation}\nn
E^{j-1}\left[M_{j-1}(\theta) (\Delta_jX -h a_{j-1}(\alpha))\right]=M_{j-1}(\theta)\left\{h (a_{j-1}-a_{j-1}(\alpha))+E^{j-1}\left[\int_j(a_s-a_{j-1})ds\right] \right\}.
\end{equation}
Now, from \cite[Lemma 4.5]{Masuda2013} and $ \sup_\theta|M(x, \theta)|\lesssim1+|x|^C$ for some $C\geq0$, we have 
\begin{align*}
E\left[\left|\frac{1}{nh}\sum_jM_{j-1}(\theta) E^{j-1}\left[\int_j(a_s-a_{j-1})ds\right]\right|\right] &\leq\frac{1}{nh} \sum_j\sqrt{E[|M_{j-1}(\theta)|^2]} \sqrt{E\left[\left|E^{j-1}\left[\int_j(a_s-a_{j-1})ds\right]\right|^2\right]}\\
&\lesssim\frac{1}{nh} \sum_j\sqrt{h \int_j E[|a_s-a_{j-1}|^2]ds}\\
&\lesssim\frac{1}{nh} \sum_j\sqrt{h \int_j E[|X_s-X_{j-1}|^2] ds} \lesssim \sqrt{h} = o(1),
\end{align*}
so the ergodic theorem gives 
\begin{equation}\nn
\frac{1}{nh} \sum_jE^{j-1}\left[M_{j-1}(\theta) (\Delta_jX -h a_{j-1}(\alpha))\right] \overset{P_0} \longrightarrow G_\infty^{\alpha}(\theta).
\end{equation}
Similarly, we see that
\begin{align*}
&E^{j-1}\left[\left[-\partial_\gamma c_{j-1}^{-2}(\gamma)\right](\Delta_jX-h a_{j-1}(\alpha))^2-h \frac{\partial_\gamma c_{j-1}^2(\gamma)}{c_{j-1}^2(\gamma)}\right]\\
&=-\partial_\gamma c_{j-1}^{-2}(\gamma) E^{j-1}\left[\left(\int_j (a_s-a_{j-1})ds+\int_j (c_{s-}-c_{j-1}) dJ_s+h (a_{j-1}-a_{j-1}(\alpha))+c_{j-1}\Delta_jJ\right)^2\right] -\frac{h\partial_\gamma c_{j-1}^2(\gamma)}{c_{j-1}^2(\gamma)}\\
&=-\partial_\gamma c_{j-1}^{-2}(\gamma) E^{j-1}\left[\zeta_{s,j-1}^2+2\zeta_{s,j-1} c_{j-1}\Delta_jJ+c_{j-1}^2(\Delta_jJ)^2 \right] -\frac{h\partial_\gamma c_{j-1}^2(\gamma)}{c_{j-1}^2(\gamma)},
\end{align*}
where $\zeta_{s,j-1}:=\int_j (a_s-a_{j-1})ds+\int_j (c_{s-}-c_{j-1}) dJ_s+h (a_{j-1}-a_{j-1}(\alpha))$.
Applying \cite[Lemma 4.5]{Masuda2013} and Burkholder's inequality, we see that
\begin{align*}
&E^{j-1}[\zeta_{s,j-1}^2]\\
&\lesssim E^{j-1}\left[\left|\int_j (a_s-a_{j-1})ds\right|^2\right]+E^{j-1}\left[\left|\int_j (c_{s-}-c_{j-1}) dJ_s\right|^2\right]+E^{j-1}\left[\left| h (a_{j-1}-a_{j-1}(\alpha))\right|^2\right]\\
&\lesssim h \int_j E^{j-1}[|X_s-X_{j-1}|^2] ds+\int_j E^{j-1}[|X_s-X_{j-1}|^2] ds+h^2(a_{j-1}-a_{j-1}(\alpha))^2\\
&\lesssim h^2(1+|X_{j-1}|^C),
\end{align*}
and $E^{j-1}[c_{j-1}^2(\Delta_jJ)^2 ]=h c_{j-1}^2$.
Hence we have $\left|E^{j-1}[\zeta_{s,j-1} c_{j-1}\Delta_jJ]\right|\lesssim h^{\frac{3}{2}}(1+|X_{j-1}|^C)$ 
by conditional Cauchy-Schwarz's inequality. It follows that
\begin{equation}\nn
\frac{1}{nh}  \sum_jE^{j-1}\left[\left[-\partial_\gamma c_{j-1}^{-2}(\gamma)\right](\Delta_jX-h a_{j-1}(\alpha))^2-h \frac{\partial_\gamma c_{j-1}^2(\gamma)}{c_{j-1}^2(\gamma)}\right]\overset{P_0} \longrightarrow G_\infty^{\gamma}(\gamma),
\end{equation}
from the ergodic theorem.
Above calculation yields that 
\begin{equation}\nn
E^{j-1}[|\Delta_jX-ha_{j-1}(\alpha)|^2] \lesssim h(1+|X_{j-1}|^C),
\end{equation} 
and obviously, this inequality is valid when we replace 2 with for any $q\geq 2$.
In the same way we can easily see that
\begin{align*}
&E^{j-1}\left[\left|\left[-\partial_\gamma c_{j-1}^{-2}(\gamma)\right](\Delta_jX-h a_{j-1}(\alpha))^2-h \frac{\partial_\gamma c_{j-1}^2(\gamma)}{c_{j-1}^2(\gamma)}\right|^2\right]\\
&\lesssim\left|-\partial_\gamma c_{j-1}^{-2}(\gamma)\right|^2 E^{j-1}[(\Delta_jX-h a_{j-1}(\alpha))^4]+h^2 \left|\frac{\partial_\gamma c_{j-1}^2(\gamma)}{c_{j-1}^2(\gamma)}\right|^2 \lesssim  h (1+|X_{j-1}|^C),
\end{align*}
so the ergodic theorem gives 
\begin{align*}
&\frac{1}{(nh)^2}  \sum_jE^{j-1}\left[\left|M_{j-1}(\theta) (\Delta_jX -h a_{j-1}(\alpha))\right|^2\right] \overset{P_0} \longrightarrow 0,\\
&\frac{1}{(nh)^2} \sum_jE^{j-1}\left[\left|\left[-\partial_\gamma c_{j-1}^{-2}(\gamma)\right](\Delta_jX-h a_{j-1}(\alpha))^2-h \frac{\partial_\gamma c_{j-1}^2(\gamma)}{c_{j-1}^2(\gamma)}\right|^2\right]\overset{P_0} \longrightarrow 0.
\end{align*}
As a result of these computations, we obtain the $\theta$-pointwise convergence
\begin{equation}\label{hm:eq.1}
\left|\tilde{G}_n(\theta)-G_\infty(\theta)\right| \overset{P_0} \longrightarrow 0.
\end{equation}
To prove the uniformity of \eqref{hm:eq.1}, 
it suffices to show the tightness, which is in turn implied by
\begin{equation}\nn
\sup_nE\left[\sup_\theta \left|\partial_\theta \tilde{G}_n(\theta)\right|\right]<\infty.
\nonumber
\end{equation}
In the case of $p_\alpha=p_\gamma=q=1$, we have
\begin{align*}
\partial_\alpha G_n^\alpha(\theta) &=\frac{1}{nh} \sum_j\left\{ \frac{\partial_\alpha^2 a_{j-1}(\alpha)}{c^2_{j-1}(\gamma)}(\Delta_jX-h a_{j-1}(\alpha)) -h \frac{(\partial_\alpha a_{j-1}(\alpha))^2}{c^2_{j-1}(\gamma)} \right\}, \\
\partial_\gamma G_n^\alpha(\theta) &=-\frac{1}{nh} \sum_j\frac{\partial_\alpha a_{j-1}(\alpha)\partial_\gamma c_{j-1}(\gamma)}{c_{j-1}^3(\gamma)}(\Delta_jX-h a_{j-1}(\alpha)), \\
\partial_\alpha G_n^\gamma(\theta)&=-\frac{2}{n} \sum_j\frac{\partial_\alpha a_{j-1}(\alpha)\partial_\gamma c_{j-1}(\gamma)}{c_{j-1}^3(\gamma)}(\Delta_jX-h a_{j-1}(\alpha)), \\
\partial_\gamma G_n^\gamma(\theta)&=\frac{2}{nh} \sum_j\left\{\frac{\partial^2_\gamma c_{j-1}(\gamma)c_{j-1}(\gamma)-3(\partial_\gamma c_{j-1}(\gamma))^2}{c_{j-1}^4(\gamma)}(\Delta_jX-h a_{j-1}(\alpha))^2  \right.\\
-&\left. h \frac{\partial^2_\gamma c_{j-1}(\gamma)c_{j-1}(\gamma)-(\partial_\gamma c_{j-1}(\gamma))^2}{c_{j-1}^2(\gamma)} \right\},
\end{align*}
and if we impose some regularity conditions  on $a$ and $c$, we can calculate the high-order derivative of $\tilde{G}_n(\theta)$ readily. 
Sobolev's inequality and Lemma \ref{MF} imply that for $q>p$,
\begin{equation}\nn
E\left[\sup_\theta \left|\partial_\theta \tilde{G}_n(\theta)\right|^q\right] \lesssim \sup_\theta E\left[\left|\partial_\theta \tilde{G}_n(\theta)\right|^q+\left|\partial_\theta^2 \tilde{G}_n(\theta)\right|^q\right]<\infty.
\end{equation}
Hence we are able to conclude that $[\{\tilde{G}_n(\theta)-G_\infty(\theta)\}_{\theta\in\Theta}]_{n\in\mathbb{N}}$ is uniformly tight (see, e.g. \cite{Kunita1997}) so that the continuous mapping theorem yields that $\sup_{\theta\in\Theta}|\tilde{G}_n(\theta)-G_\infty(\theta)|\overset{P_0}\longrightarrow0$.
Moreover, the consistency of $\hat{\theta}$ immediately follows from \cite[Theorem 5.3]{VanderVaart2000}. 
We will observe that
\begin{align*}
\sqrt{nh}G_n^\alpha(\theta_0) &=\frac{1}{\sqrt{nh}}\sum_j\frac{\partial_\alpha a_{j-1}}{c_{j-1}}\Delta_jJ +o_p(1),\\
\sqrt{nh}G_n^\gamma(\theta_0) &=\frac{2}{\sqrt{nh}}\sum_j\left\{\frac{\partial_\gamma c_{j-1}}{c_{j-1}}((\Delta_jJ)^2-h)\right\}+ o_p(1).
\end{align*}
Trivial decomposition leads to
\begin{align*}
\sqrt{nh}G_n^\alpha(\theta_0)&=\frac{1}{\sqrt{nh}}\sum_jM_{j-1} (\Delta_jX -h_n a_{j-1})\\
&=\frac{1}{\sqrt{nh}}\sum_jM_{j-1}\int_j(a_s-a_{j-1})ds \nn\\
&{}+\frac{1}{\sqrt{nh}}\sum_jM_{j-1}\int_j(c_{s-}-c_{j-1})dJ_s+ \frac{1}{\sqrt{nh}}\sum_j\frac{\partial_\alpha a_{j-1}}{c_{j-1}}\Delta_jJ. 
\end{align*}
From this, it suffices to show that $\frac{1}{\sqrt{nh}}\sum_jM_{j-1}\int_j(a_s-a_{j-1})ds$ and $\frac{1}{\sqrt{nh}}\sum_jM_{j-1}\int_j(c_{s-}-c_{j-1})dJ_s$ are $o_p(1)$.
Notice that $|M_{j-1}|\lesssim(1+|X_{j-1}|^C)$.
As in the proof of Lemma \ref{MF}, we can observe that these terms are $o_p(1)$.
%Applying \cite[Lemma 4.5]{Masuda2013}, we obtain 
%\begin{align*}
%E\left[\left|\frac{1}{\sqrt{nh}}\sum_jM_{j-1}\int_j(a_s-a_{j-1})ds\right|\right] &\leq\frac{1}{\sqrt{nh}}\sum_j\sqrt{E[|M_{j-1}|^2]}\sqrt{E\left[\left|\int_j(a_s-a_{j-1})ds\right|^2\right]} \\
%&\lesssim\frac{1}{\sqrt{nh}}\sum_j\sqrt{h \int_j E[|X_s-X_{j-1}|^2] ds}\\
%&\lesssim\frac{1}{\sqrt{nh}}\sum_jh^{\frac{3}{2}} =\sqrt{nh^2}=o(1).
%\end{align*}
%For all $q\geq2$, Burkholder's inequality gives 
%\begin{align*}
%E\left[\left|\frac{1}{\sqrt{nh}}\sum_jM_{j-1}\int_j(c_{s-}-c_{j-1})dJ_s\right|^q\right] &=(nh)^{-\frac{q}{2}} E\left[\left|\int_0^{nh} \sum_j\chi_j(s)M_{j-1}(c_{s-}-c_{j-1})dJ_s\right|^q\right]\\
%&\lesssim(nh)^{-\frac{q}{2}} (nh)^{\frac{q}{2}-1} \int_0^{nh} E\left[\left|\sum_j\chi_j(s)M_{j-1}(c_{s}-c_{j-1})\right|^q\right] ds\\
%&=\frac{1}{nh} \sum_j\int_j E[|M_{j-1}(c_s-c_{j-1})|^q] ds\\
%&\lesssim \frac{1}{nh} \sum_j\int_j \sqrt{E[|c_s-c_{j-1}|^{2q}]} ds\\
%&\lesssim \frac{1}{nh} \sum_jh^{\frac{3}{2}} =\sqrt{h} = o(1),
%\end{align*}
Hence we get
\begin{equation}\nn
\sqrt{nh}G_n^\alpha(\theta_0) =\frac{1}{\sqrt{nh}}\sum_j\frac{\partial_\alpha a_{j-1}}{c_{j-1}}\Delta_jJ +o_p(1).
\end{equation}
It is clear that
\begin{align*}
\sqrt{nh}G_n^\gamma(\theta_0)&=\frac{1}{\sqrt{nh}}\sum_{j=1}^n \left\{\left[-\partial_\gamma c_{j-1}^{-2}\right](\Delta_jX-h a_{j-1})^2-h \frac{\partial_\gamma c_{j-1}^2}{c_{j-1}^2}\right\}\\
&=\frac{1}{\sqrt{nh}}\sum_j\left[-\partial_\gamma c_{j-1}^{-2}\right]\left(\int_j(a_s-a_{j-1})ds+\int_j(c_{s-}-c_{j-1})dJ_s\right)^2\\
&+\frac{2}{\sqrt{nh}}\sum_j\left[-\partial_\gamma c_{j-1}^{-2}\right]c_{j-1}\Delta_jJ\left(\int_j(a_s-a_{j-1})ds+\int_j(c_{s-}-c_{j-1})dJ_s\right)\\
&+\frac{2}{\sqrt{nh}}\sum_j\left\{\frac{\partial_\gamma c_{j-1}}{c_{j-1}}((\Delta_jJ)^2-h)\right\}.
\end{align*}
By Assumption \ref{Smoothness}, $\partial_\gamma c_{j-1}^{-2}$ admits a polynomial majorant, so it follows that
\begin{equation}\nn
E\left[\left|\frac{1}{\sqrt{nh}}\sum_j\left[-\partial_\gamma c_{j-1}^{-2}\right]\left(\int_j(a_s-a_{j-1})ds+\int_j(c_{s-}-c_{j-1})dJ_s\right)^2\right|\right] =o(1),
\end{equation}
from Lemma \ref{MF}.
Similar calculations yield that 
\begin{align*}
&E\left[\left|\frac{1}{\sqrt{nh}}\sum_j\left[-\partial_\gamma c_{j-1}^{-2}\right]c_{j-1}\Delta_jJ\left(\int_j(a_s-a_{j-1})ds+\int_j(c_{s-}-c_{j-1})dJ_s\right)\right|\right]\\
&\lesssim\frac{1}{\sqrt{nh}}\sum_jE\left[|\partial_\gamma c_{j-1}\Delta_jJ|\left(\left|\int_j(a_s-a_{j-1})ds\right|+\left|\int_j(c_{s-}-c_{j-1})dJ_s\right|\right)\right]\\
&\lesssim\frac{1}{\sqrt{nh}}\sum_j\sqrt{E[\Delta_jJ]^2}\left(\sqrt{E\left[\left|\int_j(a_s-a_{j-1})ds\right|^2\right]}+\sqrt{E\left[\left|\int_j(c_{s-}-c_{j-1})dJ_s\right|^2\right]}\right).
\end{align*}
In the last inequality, we used the independence of increments of $J$.
By Lemma \ref{MC}, we observe that $\frac{1}{h}E[J_h^2]\to 1$, so we see that 
\begin{align*}
&\frac{1}{\sqrt{nh}}\sum_j\sqrt{E[(\Delta_jJ)^2]}\left(\sqrt{E\left[\left|\int_j(a_s-a_{j-1})ds\right|^2\right]}+\sqrt{E\left[\left|\int_j(c_{s-}-c_{j-1})dJ_s\right|^2\right]}\right)\\
&\lesssim\frac{1}{\sqrt{nh}}\sum_j\sqrt{h}({h}^\frac{3}{2}+h) \lesssim \sqrt{nh^2}=o(1).
\end{align*}
Hence we get
\begin{equation}\nn
\sqrt{nh}G_n^\gamma(\theta_0)=\frac{2}{\sqrt{nh}}\sum_j\left\{\frac{\partial_\gamma c_{j-1}}{c_{j-1}}((\Delta_jJ)^2-h)\right\}+o_p(1).
\end{equation}
We define
\begin{align*}
\sqrt{nh}\tilde{G}_n^\alpha(\theta_0)&=\frac{1}{\sqrt{nh}}\sum_j\frac{\partial_\alpha a_{j-1}}{c_{j-1}}\Delta_jJ,\\
\sqrt{nh}\tilde{G}_n^\gamma(\gamma_0)&=\frac{2}{\sqrt{nh}}\sum_j\left\{\frac{\partial_\gamma c_{j-1}}{c_{j-1}}((\Delta_jJ)^2-h)\right\}.
\end{align*} 
From Assumption \ref{Moments}, we have
\begin{align*}
&\sum_jE^{j-1}\left[\frac{\partial_\alpha a_{j-1}}{c_{j-1}}\Delta_jJ\right]=0,\\
&\sum_jE^{j-1}\left[\frac{\partial_\gamma c_{j-1}}{c_{j-1}}((\Delta_jJ)^2-h)\right]=0.
\end{align*}
The ergodic theorem and Lemma \ref{MC} give 
\begin{align*}
&\frac{1}{(nh)^2}\sum_jE^{j-1}\left[\left|\frac{\partial_\alpha a_{j-1}}{c_{j-1}}\Delta_jJ\right|^4\right] \leq  \frac{1}{(nh)^2}\sum_j\left|\frac{\partial_\alpha a_{j-1}}{c_{j-1}}\right|^4 E[J_h^4] \lesssim \frac{1}{nh}=o(1),\\ 
&\frac{1}{(nh)^2}\sum_jE^{j-1}\left[\left|\frac{\partial_\gamma c_{j-1}}{c_{j-1}}((\Delta_jJ)^2-h)\right|^4\right] \leq  \frac{1}{(nh)^2}\sum_j\left|\frac{\partial_\gamma c_{j-1}}{c_{j-1}}\right|^4  E\left[(J_h^2-h)^4\right] \lesssim \frac{1}{nh}=o(1), 
\end{align*}
so the Lindeberg condition holds.
Furthermore we get 
\begin{align*}
&E^{j-1}\left[\frac{\partial_{\alpha_k} a_{j-1}\partial_{\alpha_l} a_{j-1}}{c_{j-1}^2}(\Delta_jJ)^2\right] =\frac{\partial_{\alpha_k} a_{j-1} \partial_{\alpha_l} a_{j-1}}{c_{j-1}^2} E[J_h^2] = \frac{\partial_{\alpha_k} a_{j-1} \partial_{\alpha_l} a_{j-1}}{c_{j-1}^2} h,\\
&E^{j-1} \left[\frac{\partial_{\gamma_k} c_{j-1}\partial_{\gamma_l} c_{j-1}}{c_{j-1}^2}((\Delta_jJ)^2-h)^2\right]= \frac{\partial_{\gamma_k} c_{j-1}\partial_{\gamma_l} c_{j-1}}{c_{j-1}^2} E[(J_h^2-h)^2] = \frac{\partial_{\gamma_k} c_{j-1}\partial_{\gamma_l} c_{j-1}}{c_{j-1}^2} \left\{E[J_h^4]+o(h)\right\},\\
&E^{j-1}\left[\frac{\partial_{\alpha_k} a_{j-1}\partial_{\gamma_l} c_{j-1}}{c_{j-1}^2} \Delta_jJ ((\Delta_jJ)^2-h) \right] =\frac{\partial_{\alpha_k} a_{j-1}\partial_{\gamma_l} c_{j-1}}{c_{j-1}^2} E[J_h^3].
\end{align*}
Finally, we apply the ergodic theorem to derive
\begin{align*}
&\frac{1}{nh}\sum_jE^{j-1}\left[\frac{\partial_{\alpha_k} a_{j-1}\partial_{\alpha_l} a_{j-1}}{c_{j-1}^2}(\Delta_jJ)^2\right]  \overset{P_0}\longrightarrow \int \frac{\partial_{\alpha_k} a(x,\alpha_0) \partial_{\alpha_l} a(x,\alpha_0)}{c^2(x,\gamma_0)} \pi_0(dx),\\
&\frac{1}{nh}\sum_jE^{j-1} \left[\frac{\partial_{\gamma_k} c_{j-1}\partial_{\gamma_l} c_{j-1}}{c_{j-1}^2}((\Delta_jJ)^2-h)^2\right] \overset{P_0}\longrightarrow \int \frac{\partial_{\gamma_k} c(x,\gamma_0)\partial_{\gamma_l} c(x,\gamma_0)}{c^2(x,\gamma_0)} \pi_0(dx) \int z^4 \nu_0(dz),\\
&\frac{1}{nh}\sum_jE^{j-1}\left[\frac{\partial_{\alpha_k} a_{j-1}\partial_{\gamma_l} c_{j-1}}{c_{j-1}^2} \Delta_jJ ((\Delta_jJ)^2-h) \right] \overset{P_0}\longrightarrow \int \frac{\partial_{\alpha_k} a(x,\alpha_0)\partial_{\gamma_l} c(x,\gamma_0)}{c^2(x,\gamma_0)} \pi_0(dx) \int z^3 \nu_0(dz),
\end{align*}
with which the martingale central limit theorem completes the proof.
\end{proof}

Applying Taylor's theorem to $\tilde{G}_n(\theta_0)$, we get
\begin{equation}\nn
\tilde{G}_n(\theta_0) = -\int_0^1\partial_\theta\tilde{G}(\hat{\theta}+u(\theta_0-\hat{\theta}))du\sqrt{nh}(\hat{\theta}-\theta_0).
\end{equation}
Note that by the consistency of $\alpha$ and $\gamma$, we can consider $\tilde{G}_n(\hat{\theta})=0$ a.s., for large enough $n$.

\begin{Lem}\label{AN3}
If Assumptions \ref{Sampling design}-\ref{Nondegeneracy} hold, we have
\begin{align*}
&\sup_{|\theta|\leq\epsilon_n}\left|-\partial_\theta \tilde{G}_n(\theta_0+\theta)-\mathcal{I}(\theta_0)\right|\longrightarrow0,\quad where\quad\epsilon_n\to0\\
&\sqrt{nh}(\hat{\theta}-\theta_0) \overset{\mathcal{L}} \longrightarrow \mathcal{N}(0,(\mathcal{I}(\theta_0)^{-1})^T\Sigma_{22}\mathcal{I}(\theta_0)^{-1}). 
\end{align*}
\end{Lem} 

\begin{proof}
We may set $p_\alpha=p_\gamma=1$. Define the $2\times 2$-valued matrix $\mathcal{I}(\theta)$ such that
\begin{equation}\nn
\mathcal{I}(\theta)= \begin{pmatrix}
\mathcal{I}^{(\alpha,\alpha)}(\theta)&\mathcal{I}^{(\alpha,\gamma)}(\theta)\\
0&\mathcal{I}^{(\gamma,\gamma)}(\theta)
\end{pmatrix},
\end{equation}
where $\mathcal{I}^{(\alpha,\alpha)}(\theta)$, $\mathcal{I}^{(\alpha,\gamma)}(\theta)$ and $\mathcal{I}^{(\gamma,\gamma)}(\theta)$ are defined by
\begin{align*}
&\mathcal{I}^{(\alpha,\alpha)}(\theta)=\int \left\{\frac{\partial^2_\alpha a(x,\alpha)}{c^2(x,\gamma)}(a(x,\alpha)-a(x,\alpha_0))+\frac{(\partial_\alpha a(x,\alpha))^2}{c(x,\gamma)^2}\right\} \pi_0(dx),\\
&\mathcal{I}^{(\alpha,\gamma)}(\theta)=\int \frac{\partial_\alpha a(x,\alpha)\partial_\gamma c(x,\gamma)}{c^3(x,\gamma)}(a(x,\alpha_0)-a(x,\alpha)) \pi_0(dx),\\
&\mathcal{I}^{(\gamma,\gamma)}(\theta)=4\int \frac{(\partial_\gamma c(x,\gamma))^2 }{c^2(x,\gamma)} \pi_0(dx).
\end{align*}
As in the previous lemma, we can prove
\begin{equation}\nn
-\partial_\theta \tilde{G}_n(\theta)\overset{P_0}\longrightarrow \mathcal{I}(\theta),\quad\mbox{for all} \ \theta.
\end{equation}
By Assumption \ref{Smoothness}, it immediately follows that for all $k\in\left\{1,2,3,4\right\}$, $\partial_\theta^k \tilde{G}_n(\theta)$ can be decomposed as
\begin{equation}\nn
\partial_\theta^k \tilde{G}_n(\theta)=\frac{1}{nh} \sum_j\left\{M_{j-1}^{(1,k)}(\theta)(\Delta_jX-h a_{j-1}(\alpha))^2+M_{j-1}^{(2,k)}(\theta)(\Delta_jX-h a_{j-1}(\alpha))+h M_{j-1}^{(3,k)}(\theta)\right\},
\end{equation}
where $M_{j-1}^{(1,k)}$, $M_{j-1}^{(2,k)}$ and $M_{j-1}^{(3,k)}$ are functions of $X_{t_{j-1}}$ at most polynomial growth uniformly in $\theta$. 
Hence the Sobolev's inequality implies that $\left[\left\{-\partial_\theta \tilde{G}_n(\theta)-\mathcal{I}(\theta)\right\}_{\theta\in\Theta}\right]_{n\in\mathbb{N}}$ is uniformly tight and the continuous mapping theorem gives
\begin{equation}\nn
\sup_{|\theta|\leq\epsilon_n}\left|-\partial_\theta \tilde{G}_n(\theta_0+\theta)-\mathcal{I}(\theta_0)\right|\longrightarrow0,\quad where\quad\epsilon_n\to 0.
\end{equation}
Further, the continuity of $\mathcal{I}(\theta)$ and the consistency of $\hat{\theta}$ give 
\begin{equation}\nn
-\int_0^1\partial_\theta\tilde{G}(\hat{\theta}+u(\theta_0-\hat{\theta}))du \overset{P_0} \longrightarrow \mathcal{I}(\theta_0).
\end{equation}
Assumption \ref{Nondegeneracy} ensures that $ \lim_{n\to\infty}P\left(\left|-\int_0^1\partial_\theta\tilde{G}(\hat{\theta}+u(\theta_0-\hat{\theta}))du\right|>0\right)=1$, hence we can suppose that $-\int_0^1\partial_\theta\tilde{G}(\hat{\theta}+u(\theta_0-\hat{\theta}))du$ is invertible for all $n$ large enough. Hence, applying Slutsky's lemma, we have the desired result. 
\end{proof}

Obviously, it follows from Lemma \ref{AN3} that $-\partial_\theta \tilde{G}_n(\hat{\theta})$ 
can serve as a consistent estimator of $\mathcal{I}(\theta_0)$. In the same way, we could provide a consistent estimator of the asymptotic variance of $\hat{\theta}$, making it possible to construct a confidence region. 

We introduce the following function space:
\begin{align*}
\mathcal{K}_{2}
&=\Biggl\{f=(f_{k}):\mbbr\to\mbbr^{q}\, \Bigg|\, \text{$f$ is of class $C^{2}$}, \quad 
\frac{1}{h}\max_{1 \leq j\leq n}E\left[\left|\partial f(\delta_j)\right|^2\right]=O(1),
\nonumber \\
& {}\qquad \frac{1}{h}\max_{1\leq j\leq n} \sup_{u\in[0,1]}E\left[\left|\partial f(\Delta_jJ+u(\delta_j-\Delta_jJ))\right|^2\right]
=O(1),
\\
&{}\qquad\text{and} \quad 
\forall K>0,\quad 
\max_{1 \leq j\leq n}\sup_{u\in[0,1]}E
\left[\left|\partial^2 f(\hat{\delta}_j+u(\delta_j-\hat{\delta}_j))\right|^K\right]
=O(1)
\Biggr\}.
\end{align*}

By use of this class we can prove:
\begin{Lem}\label{BCL}
Suppose that Assumptions \ref{Sampling design}-\ref{Identifiability} hold and that $\varphi\in\mathcal{K}_2$.
Then we have the stochastic expansion:
\begin{equation*}
\sqrt{nh}\bigg(\frac{1}{nh}\sum_j\varphi(\hat{\delta}_j)-\nu_0(\varphi)\bigg)=u_n+\frac{1}{nh}\sum_j 
(\partial \varphi(\delta_j) \otimes \partial_\gamma(c_{j-1}^{-1}))c_{j-1}\D_{j}J[\hat{w}_n]+o_p(1),
\end{equation*}
where we also have $\frac{1}{nh}\sum_j (\partial \varphi(\delta_j) \otimes \partial_\gamma(c_{j-1}^{-1}))c_{j-1}\D_{j}J=O_p(1)$.
\end{Lem}

\begin{proof}
First we decompose the left-hand side as
\begin{align*}
&\sqrt{nh}\left\{\frac{1}{nh}\sum_{j=1}^n\varphi(\hat{\delta}_j)-\nu_0(\varphi)\right\}\\
 &=\sqrt{nh}\left\{\frac{1}{nh}\sum_j\left[\varphi(\hat{\delta}_j)-\varphi(\delta_j)\right]\right\} +\sqrt{nh}\left\{\frac{1}{nh}\sum_j\left[\varphi(\delta_j)-\varphi(\Delta_jJ)\right]\right\}+u_n\\
&=: b_n^{(1)}+b_n^{(2)}+u_n.
\end{align*}
Let us first prove $b_n^{(2)}=o_p(1)$.
Applying Taylor's theorem, we see that
\begin{equation}\nn
b_n^{(2)}=\frac{1}{\sqrt{nh}} \sum_j\left[\int_0^1\partial\varphi(\Delta_jJ+u(\delta_j-\Delta_jJ))du\right](\delta_j -\Delta_jJ).
\end{equation}
By definition of $\delta_j$, it follows that
\begin{equation}\nn
\Delta_jJ-\delta_j =c_{j-1}^{-1}(c_{j-1}\Delta_jJ-\Delta_jX-ha_{j-1}) =c_{j-1}^{-1}\left(\int_j(a_s-a_{j-1})ds+\int_j(c_{s-}-c_{j-1})dJ_s\right).
\end{equation}
%From Jensen's inequality and Lemma 4.5 in \cite{Masuda2013}, we have 
%\begin{align*}
%E\left[\left|\int_j(a_s-a_{j-1})ds\right|^2\right]&\leq h \int_j E\left[\left|a_s-a_{j-1}\right|^2\right]ds\\
%&\lesssim h \int_j E\left[\left|X_s-X_{t_{j-1}}\right|^2\right]ds\\
%&\lesssim h^2 \int_j E\left[|g_{j-1}|^2\right] ds \lesssim  h^3.
%\end{align*}
%Burkholder's inequality gives 
%\begin{align*}
%E\left[\left|\int_j(c_{s-}-c_{j-1})dJ_s\right|^2\right]&\lesssim\int_j E\left[\left|(c_s-c_{j-1})%\right|^2\right] ds \lesssim h^2,
%\end{align*}
As in the proof of Lemma \ref{AN2}, we have
\begin{equation}\label{hm:add.eq6}
E\left[\left|\Delta_jJ-\delta_j\right|^q\right] \lesssim h^2,
\end{equation}
for all $q \geq 2$.
Applying Cauchy-Schwarz's inequality we get 
\begin{align*}
E[|b_n^{(2)}|]&\leq\frac{1}{\sqrt{n}} E\left[\sum_j\frac{1}{\sqrt{h}}\left|\int_0^1\partial\varphi(\Delta_jJ+u(\delta_j-\Delta_jJ))du\right||\Delta_jJ-\delta_j|\right]\\
&\leq\frac{1}{\sqrt{n}} \sum_j\sqrt{\frac{1}{h}E\left[\left|\int_0^1\partial\varphi(\Delta_jJ+u(\delta_j-\Delta_jJ))du\right|^2\right]} \sqrt{E[\left|\Delta_jJ-\delta_j\right|^2]}\\
&\leq\frac{1}{\sqrt{n}}\max_{1\leq j \leq n} \sqrt{\frac{1}{h}\sup_{u\in[0,1]}E\left[\left|\partial \varphi(\Delta_jJ+u(\delta_j-\Delta_jJ))\right|^2\right]} \sum_j\sqrt{E[\left|\Delta_jJ-\delta_j\right|^2]}\\
&\lesssim\sqrt{nh^2} =o(1).
\end{align*}

Next we turn to $b^{(1)}_{n}$. 
By Taylor's theorem, we have
\begin{align*} 
b_n^{(1)}&=\frac{1}{\sqrt{nh}}\sum_j\left[\partial\varphi(\delta_j)(\hat{\delta}_j-\delta_j) \right] 
+\frac{1}{2\sqrt{nh}}\sum_j\left[\int_0^1 \int_0^1 v \partial^2\varphi(\delta_j+uv(\hat{\delta}_j-\delta_j))dvdu
(\hat{\delta}_j-\delta_j)^2\right].\\
&=:b_n^{(1,1)}+b_n^{(1,2)}.
\end{align*}
For notational convenience, we denote by $R(x)$ a generic matrix-valued function defined on $\mbbr\times\Theta$ for which 
there exists a constant $C\ge 0$ such that $\sup_{\theta}|R(x,\theta)|\le C(1+|x|^{C})$ for every $x$; 
the argument $\theta$ is omitted from the notation, and the specific form of $R_{j-1}$ appearing below may vary from line to line. 
From the definition of $\hat{\delta}_j$ and $\delta_j$, 
\begin{align}
\hat{\delta}_j-\delta_j
&=\hat{c}_{j-1}^{-1}(\Delta_jX-h\hat{a}_{j-1})-c_{j-1}^{-1}(\Delta_jX-ha_{j-1}) \nn\\
&=(\hat{c}_{j-1}^{-1}-c_{j-1}^{-1})\Delta_jX -h(\hat{\eta}_{j-1}-\eta_{j-1}).
%\label{hm:add.eq4}
\end{align}
Again applying Taylor's theorem, we obtain
\begin{align}
|\hat{\delta}_j-\delta_j|^2 
&\lesssim\frac{1}{nh}\left[
\left(\sup_{\gam}|\partial_\gamma c_{j-1}^{-1}(\gam)|\right)^2|\hat{w}|^2|\Delta_jX|^2
+h^2\left(\sup_{\theta}|\partial_\theta \eta_{j-1}(\theta)|\right)^2|\hat{v}|^2\right]
\nn\\
&\lesssim \frac{1}{nh}\left(|\hat{w}|^2|\Delta_jX|^2+h^2|\hat{v}|^2\right)|R_{j-1}|
\nn\\
&\lesssim \frac{1}{nh}\left(|\Delta_jX|^2+h^2\right)|R_{j-1}||\hat{v}|^2.
\label{hm:add.eq7}
\end{align}
A similar argument to the proof of Lemma \ref{MF} gives the estimate $E[|R_{j-1}|E^{j-1}[|\Delta_jX|^q]]\lesssim h$ for all $q\geq2$. 
By means of these estimates and H\"older's inequality we can deduce that, 
for sufficiently large $p\ge 2$ and sufficiently small $q>1$,
\begin{align*}
|b_n^{(1,2)}|&\lesssim\frac{1}{\sqrt{nh}} \sum_j\left|\int_0^1 \int_0^1 v 
\partial^2\varphi(\delta_j+uv(\hat{\delta}_j-\delta_j))dvdu\right| |\hat{\delta}_j-\delta_j|^2\\
&\lesssim\frac{1}{\sqrt{nh}} \frac{1}{nh} |\hat{v}|^2 \sum_j
\bigg|\int_0^1 \int_0^1 v \partial^2\varphi(\delta_j+uv(\hat{\delta}_j-\delta_j))dvdu\bigg|
(|\Delta_jX|^2+h^2)|R_{j-1}| \nn\\
&\leq \frac{1}{\sqrt{nh}}\frac{1}{h} |\hat{v}|^2 \left(\frac{1}{n}\sum_j\left|\int_0^1 \int_0^1 v 
\partial^2\varphi(\delta_j+uv(\hat{\delta}_j-\delta_j))dvdu\right|^p \right)^{\frac{1}{p}}\\
&\times\left[\frac{1}{n}\sum_j\left\{
%\sup_\theta \left(|\partial_\gamma c_{j-1}^{-1}(\gamma)|^2\vee|\partial_\theta \eta_{j-1}(\theta)|^2\right)
|R_{j-1}|
\left(|\Delta_jX|^2+h^2\right)\right\}^{\frac{p}{p-1}}\right]^{\frac{p-1}{p}}\\
&\lesssim \frac{1}{\sqrt{nh}}\frac{1}{h}\times O_{p}(1)\times\left\{\left(\frac{1}{n}\sum_j|\Delta_jX|^\frac{2pq}{p-1}\right)^\frac{p-1}{pq}
\times O_{p}(1)+O_p(h^{2})\right\}
\nn\\
&\lesssim \frac{1}{\sqrt{nh^{1+\epsilon_{0}}}}h^{\epsilon_{0}/2+\frac{p-1}{pq}-1} 
\times O_{p}(1)
\lesssim O_{p}\bigg(\frac{1}{\sqrt{nh^{1+\epsilon_{0}}}}\bigg)=o_{p}(1).
\end{align*}

As for $b_n^{(1,1)}$, we first observe that
\begin{equation}\nn
\hat{c}_{j-1}^{-1}-c_{j-1}^{-1}
=\frac{1}{\sqrt{nh}} \partial^{T}_\gamma(c_{j-1}^{-1}) \hat{w}
+\frac{1}{2nh}\hat{w}^T\left[\int_0^1\int_0^1v\partial_\gamma^{\otimes2}
(c_{j-1}^{-1})(\gam_0+uv(\hat{\gam}-\gam_{0}))dvdu\right]\hat{w}.
\end{equation}
In a similar way to the estimate of $|b_n^{(1,2)}|$, it follows from 
the definition of $\mathcal{K}_2$, the tightness of $(\hat{w})$, and Cauchy-Schwarz's inequality that
\begin{align*}
&\left|{(nh)}^{-\frac{3}{2}}\sum_j\partial \varphi(\delta_j)\Delta_jX\hat{w}^T
\left[\int_0^1\int_0^1v\partial_\gamma^{\otimes2}(c_{j-1}^{-1})(\gam_0+uv(\hat{\gam}-\gam_{0}))dvdu\right] \hat{w}\right|
\nn\\
&\lesssim 
{(nh)}^{-\frac{3}{2}}\sum_j \left|\partial \varphi(\delta_j)\right| |\Delta_jX||R_{j-1}|\times O_{p}(1)
\nn\\
&\lesssim\frac{1}{\sqrt{nh}}
\bigg(\frac{1}{nh}\sum_j \left|\partial \varphi(\delta_j)\right|^{2}\bigg)^{1/2}
\bigg(\frac{1}{nh}\sum_{j}|\Delta_jX|^{2}|R_{j-1}|\bigg)^{1/2}\times O_{p}(1)
\nn\\
&\lesssim O_{p}\bigg(\frac{1}{\sqrt{nh}}\bigg)=o_{p}(1).
\end{align*}
We also have
\begin{equation}\nn
\left|\sqrt{\frac{h}{n}}\sum_j\partial\varphi(\delta_j)(\hat{\eta}_{j-1}-\eta_{j-1})\right|\leq\frac{1}{n}\sum_j
|\partial\varphi(\delta_j)|\left|\int_0^1\partial_\theta\eta_{j-1}(\theta_0+u(\hat{\theta}-\theta_0))du\right||\hat{v}|=o_{p}(1).
\end{equation}
We thus get
\begin{equation}
b_n^{(1,1)}=\bigg\{\frac{1}{nh}\sum_j \Delta_jX 
\bigg(\partial \varphi(\delta_j) \otimes \partial_\gamma(c_{j-1}^{-1})\bigg)\bigg\}[\hat{w}]+o_p(1)
=:\mu_{n}[\hat{w}]+o_{p}(1).
\label{hm:add.eq5}
\end{equation}
It remains to take a closer look at $\mu_{n}\in\mbbr^{q}\otimes\mbbr^{p_{\gam}}$. 
Substitute the expression
\begin{equation}
\D_{j}X=\int_{j}a_{s}ds+\int_{j}(c_{s-}-c_{j-1})dJ_{s}+c_{j-1}\D_{j}J
\nonumber
\end{equation}
into \eqref{hm:add.eq5} and observe that
\begin{align}
\bigg|
\frac{1}{nh}\sum_j \int_{j}a_{s}ds
\bigg(\partial \varphi(\delta_j) \otimes \partial_\gamma(c_{j-1}^{-1})\bigg)
\bigg|&\lesssim 
\frac{1}{n}\sum_j |\partial \varphi(\delta_j)||R_{j-1}|
\bigg(\frac{1}{h}\int_{j}|a_{s}|ds\bigg)
\nonumber\\
&\lesssim
\bigg(\frac{1}{n}\sum_j |\partial \varphi(\delta_j)|^{2}\bigg)^{1/2}
\bigg\{\frac{1}{n}\sum_{j}|R_{j-1}|
\bigg(\frac{1}{h}\int_{j}|a_{s}|^{2}ds\bigg)\bigg\}^{1/2}
\nn\\
&\lesssim O_{p}(\sqrt{h}),
\nn
\end{align}
and similarly that, by using Burkholder's inequality (conditional on $\mathcal{F}_{t_{j-1}}$),
\begin{align}
& \bigg|
\frac{1}{nh}\sum_j \int_{j}(c_{s-}-c_{j-1})dJ_{s}
\bigg(\partial \varphi(\delta_j) \otimes \partial_\gamma(c_{j-1}^{-1})\bigg)
\bigg|\nn\\
&\lesssim 
\frac{1}{n}\sum_j |\partial \varphi(\delta_j)||R_{j-1}|
\bigg(\frac{1}{\sqrt{h}}\int_{j}\frac{1}{\sqrt{h}}(c_{s-}-c_{j-1})dJ_s\bigg)
\nonumber\\
&\lesssim
\bigg(\frac{1}{n}\sum_j |\partial \varphi(\delta_j)|^{2}\bigg)^{1/2}
\bigg\{\frac{1}{n}\sum_{j}|R_{j-1}|
\bigg(\frac{1}{\sqrt{h}}\int_{j}\frac{1}{\sqrt{h}}(c_{s-}-c_{j-1})dJ_s\bigg)^{2}\bigg\}^{1/2}
\nn\\
&\lesssim O_{p}(\sqrt{h}).
\nn
\end{align}
Therefore $\mu_{n}=\frac{1}{nh}\sum_j 
(\partial \varphi(\delta_j) \otimes \partial_\gamma(c_{j-1}^{-1}))c_{j-1}\D_{j}J+o_p(1)$ and we also get
\begin{align*}
& E\left[\left|\frac{1}{nh}\sum_j 
(\partial \varphi(\delta_j) \otimes \partial_\gamma(c_{j-1}^{-1}))c_{j-1}\D_{j}J\right|\right] \nn\\
&\leq\left(\frac{1}{n}\sum_j\frac{1}{h}E\left[|\partial\varphi(\delta_j)|^2\right]\right)^{1/2}\left(\frac{1}{n}\sum_jE\left[|R_{j-1}|^2\frac{1}{h}E[|\Delta_j J|^2]\right]\right)^{1/2}=O(1),
\end{align*}
hence the proof is complete.
\end{proof}

\subsection{Proof of Theorem \ref{Bias correction}}
In order to obtain \eqref{SE}, we first show that actually $\varphi\in\mathcal{K}_2$ and $\zeta\in\mathcal{K}_1\cap\mathcal{K}_2$ (recall the notation $\zeta(z)=z\p\vp(z)$).
As in the proof of Lemma \ref{AN}, it follows that $\zeta\in\mathcal{K}_1$.
From the proof of Lemma \ref{MC} and Lemma \ref{BCL}, for all $C\geq2$, we have
\begin{equation}\nn
\max_{1\leq j\leq n}E\left[|\Delta_jJ-\delta_j|^C\right]=O(h^2),\quad\max_{1\leq j\leq n}E\left[|\Delta_jJ|^C\right]=O(h).
\end{equation}
Moreover, \cite[Theorem 2.7]{Masuda2013} and \eqref{hm:add.eq7} give
\begin{equation}
E\left[\left|\delta_j-\hat{\delta}_j\right|^C\right]\lesssim(nh)^{-\frac{C}{2}}
E\left[\left(|\Delta_jX|^C+h^C\right)|R_{j-1}||\hat{v}|^C\right]=O\left((nh)^{-\frac{C}{2}}h^{1-a}\right),
\nonumber
\end{equation}
for any $a\in(0,1)$. Hence the Chebyshev's inequality yields that
\begin{equation}\nn
\max_{1\leq j \leq n} \sup_{u\in[0,1]}\left\{P\left(|\Delta_jJ+u(\delta_j-\Delta_jJ)|> M\right)\vee P\left(|\hat{\delta}_j+u(\delta_j-\hat{\delta_j} )|> M\right)\right\}=O(h).
\end{equation}
We will use these estimates without notice below. 
%It is straightforward to see that the estimates \eqref{hm:add.eq1} to \eqref{hm:add.eq3} 
%remain valid even when $\vp$ is replaced by the mappings 
%$z\mapsto z\varphi(z),z^2\varphi(z),\varphi_{1}\vp(z),\dots,\vp_{q}\vp(z)$, entailing that all of them belong to $\mathcal{K}_1$. 
By the condition on $\partial \varphi$,  we have
\begin{align*}
&\sup_{u\in[0,1]} E\left[\left|\partial \varphi(\Delta_jJ+u(\delta_j-\Delta_jJ))\right|^2\right]\\
&\lesssim E\left[|\Delta_jJ|^2+|\delta_j-\Delta_jJ|^2+|\Delta_jJ|^{2(1+C_1)}+|\delta_j-\Delta_jJ|^{2(1+C_1)}\right] =O(h).
\end{align*}
In the same way as above, we also obtain $E\left[\left|\partial \varphi(\hat{\delta}_j)\right|^2\right]=O(h)$.
By Assumption \ref{Moment fitting function}, for all $K>0$, there exists a constant $C\geq 2$ such that
\begin{align*}
\left|\partial^2 \varphi(\hat{\delta}_j+u(\delta_j-\hat{\delta}_j))\right|^K&\lesssim1+\left|\hat{\delta}_j\right|^C+ \left|\delta_j-\hat{\delta}_j \right|^C\\
&\lesssim1+ \left|\delta_j-\hat{\delta}_j \right|^C+\left|\Delta_jJ-\delta_j\right|^C+|\Delta_jJ|^C,
\end{align*}
so it is straightforward that
\begin{align*}
&\max_{1 \leq j\leq n}\sup_{u\in[0,1]}E_0\left[\left|\partial^2 \varphi(\hat{\delta}_j+u(\delta_j-\hat{\delta}_j))\right|^K\right]\\
&\lesssim1+\max_{1 \leq j\leq n}E_0\left[\left|\delta_j-\hat{\delta}_j \right|^C+\left|\Delta_jJ-\delta_j\right|^C+|\Delta_jJ|^C\right]=O(1).
\end{align*}
Hence $\varphi\in\mathcal{K}_2$; similarly $\zeta\in\mathcal{K}_2$. 

Now we have
$\del_{j}-\D_{j}J = c_{j-1}^{-1}\int_{j}(a_{s}-a_{j-1})ds + c_{j-1}^{-1}\int_{j}(c_{s}-c_{j-1})dJ_s$; 
then, $E[E^{j-1}[|\del_{j}-\D_{j}J|^{2}]]\lesssim h^{2}$. 
Plugging-in the expression 
$\p\vp(\del_{j})=\p\vp(\D_{j}J) + (\del_{j}-\D_{j}J)\int_{0}^{1}\p^{2}\vp(\D_{j}J+u(\del_{j}-\D_{j}J))du$ 
and then applying analogous estimates under Assumption \ref{Moment fitting function} as before, 
we can deduce that 
\begin{align*}
&\left|\frac{1}{nh}\sum_j 
\left((\del_{j}-\D_{j}J)\int_{0}^{1}\p^{2}\vp(\D_{j}J+u(\del_{j}-\D_{j}J))du \otimes \partial_\gamma(c_{j-1}^{-1})\right)c_{j-1}\D_{j}J\right|\\
&\leq \left(\frac{1}{nh^2}\sum_j|\del_{j}-\D_{j}J|^2\right)^{1/2}\times
\left(\frac{1}{n}\sum_j\left|\int_{0}^{1}\p^{2}\vp(\D_{j}J+u(\del_{j}-\D_{j}J))du\right|^2\left|R_{j-1}\right|^2|\Delta_j J|^2\right)^{1/2}\\
&\lesssim O_p(1)\times
\left(\frac{1}{n}\sum_j\left|R_{j-1}\right|^2|\Delta_j J|^2(1+|\Delta_j J-\delta_j|^C+|\Delta_j J|^C)\right)^{1/2}=o_p(1).
\end{align*}
It follows from \cite[Theorem 1]{Lopez2008} and Lemma \ref{BCL} that 
under the present assumptions about $\zeta$ we have 
$\frac{1}{h}E[\zeta(\D_{j}J)]=\nu_0(\zeta)+o(1)=\frac{1}{nh}\sum_{j}\zeta(\hat{\del}_{j})+o_{p}(1)$. 
Therefore,
\begin{align}
\mu_{n}&=\frac{1}{nh}\sum_j \left\{\zeta(\D_{j}J) \otimes \partial_\gamma(c_{j-1}^{-1})\right\}c_{j-1}+o_{p}(1)
\nonumber\\
&=-\frac{1}{h}E[\zeta(\D_{j}J)] \otimes 
\bigg(\frac{1}{n}\sum_j \frac{\partial_\gamma c_{j-1}}{c_{j-1}}\bigg) \nn\\
&{}\qquad
+\frac{1}{nh}\sum_j \left\{(\zeta(\D_{j}J)-E[\zeta(\D_{j}J)]) \otimes \partial_\gamma(c_{j-1}^{-1})\right\}c_{j-1}+o_{p}(1)
\nn\\
&{}=-\frac{1}{h}E[\zeta(\D_{j}J)] \otimes 
\bigg(\frac{1}{n}\sum_j \frac{\partial_\gamma c_{j-1}}{c_{j-1}}\bigg)+o_{p}(1)
\nn\\
&=-\bigg(\frac{1}{nh}\sum_{j}\zeta(\hat{\del}_{j})\bigg) \otimes 
\bigg(\frac{1}{n}\sum_j \frac{\partial_\gamma \hat{c}_{j-1}}{\hat{c}_{j-1}}\bigg)
+o_{p}(1),
\nn
\end{align}
where we used the martingale central limit theorem together with Burkholder's inequality for the third equality. 
%Applying Lemma \ref{MC} and Lemma \ref{BCL} it also follows that $\nu_0(\zeta) = \frac{1}{nh}\sum_{j}\zeta(\hat{\del}_{j})+o_{p}(1)$.
Thus the proof is complete.
%To do this, it suffices to observe that 
%the convergence $\frac{1}{h}E[\zeta(\D_{j}J)] \to \nu_{0}(\zeta)$ by Lemma \ref{Rate of gap}; 
%the estimate $\frac{1}{\sqrt{nh}}\sum_{j}(\zeta(\hat{\del}_{j}-\zeta(\D_{j}J))=O_{p}(1)$ 
%by making use of \eqref{hm:add.eq6} and \eqref{hm:add.eq7}; 
%$\frac{1}{\sqrt{nh}}\sum_{j}(\zeta(\D_{j}J)-E[\zeta(J_{h})])=O_{p}(1)$ by Burkholder's inequality 
%(note that $E[|\zeta(J_{h})|^{2}]\lesssim h$). 

%%%

\subsection{Proof of Theorem \ref{JAN1}}

From Lemmas \ref{AN}, \ref{AN2} and \ref{AN3}, it suffices to show that
\begin{align*}
&\frac{1}{nh}\sum_jE^{j-1}\left[(\varphi_{k}(\Delta_jJ)-E[\varphi_{k}(\Delta_jJ)])\left(\frac{\partial_{\alpha_l}a_{j-1}}{c_{j-1}}\Delta_jJ\right)\right] \overset{P_0}\longrightarrow  \int \varphi_{k}(z)z \nu_0(dz) \int \frac{\partial_{\alpha_l}a(x,\alpha_0)}{c(x,\gamma_0)}  \pi_0(dx),\\
&\frac{1}{nh}\sum_jE^{j-1}\left[(\varphi_{k}(\Delta_jJ)-E[\varphi_{k}(\Delta_jJ)])\left(\frac{\partial_{\gamma_l} c_{j-1}}{c_{j-1}}((\Delta_jJ)^2-h)\right)\right] \overset{P_0}\longrightarrow \int \varphi_{k}(z)z^2 \nu_0(dz)  \int \frac{\partial_{\gamma_l} c(x,\gamma_0)}{c(x,\gamma_0)} \pi_0(dx).
\end{align*}
Assumption \ref{Moments} yields that
\begin{equation}\nn
E^{j-1}\left[(\varphi_{k}(\Delta_jJ)-E[\varphi_{k}(\Delta_jJ)])\left(\frac{\partial_{\alpha_l}a_{j-1}}{c_{j-1}}\Delta_jJ\right)\right] = \frac{\partial_{\alpha_l}a_{j-1}}{c_{j-1}} E[\varphi_{k}(J_h)J_h]
\end{equation}
Similarly, we have
\begin{equation}\nn
E^{j-1}\left[(\varphi_{k}(\Delta_jJ)-E[\varphi_{k}(\Delta_jJ)])\left(\frac{\partial_{\gamma_l} c_{j-1}}{c_{j-1}}((\Delta_jJ)^2-h)\right)\right] =  \frac{\partial_{\gamma_l} c_{j-1}}{c_{j-1}} \left\{E[\varphi_{k}(J_h)J_h^2]-h E[\varphi_{k}(J_h)]\right\}.
\end{equation}
From the proof of Lemma \ref{AN} we can readily observe that $z\varphi,z^2\varphi\in\mathcal{K}_1$.
Hence the ergodic theorem and Lemma \ref{Rate of gap} lead to the desired result.

\subsection{Proof of Corollary \ref{JAN2}}

For the construction of asymptotic variance, we define the function space:
\begin{align*}
\mathcal{K}_3
&=
\Biggl\{f=(f_{k}):\mbbr\to\mbbr^{q}\, \Bigg|\, \text{$f$ is of class $C^{1}$}, \quad 
\max_{1\leq j\leq n} \sup_{u\in[0,1]}E \left[\left|\partial f(\Delta_jJ+u(\delta_j-\Delta_jJ))\right|^2\right]=o(1),
\nonumber \\
&{}\qquad \text{and}\quad 
\frac{1}{nh^{2}}\max_{1\leq j\leq n} \sup_{u\in[0,1]}
E \left[\left|\partial f(\hat{\delta}_j+u(\delta_j-\hat{\delta}_j))\right|^2\right]=o(1)\Biggr\}.
\end{align*}

The following lemma gives sufficient conditions for a given function to belong to $\mathcal{K}_3$.
\begin{Lem}\label{K3}
Assume that an $\mathbb{R}^q$-valued or $\mathbb{R}^q\otimes\mathbb{R}^q$-valued function $f$ is differentiable and there exist nonnegative constant $D$ such that $\limsup_{z\to 0}\frac{1}{|z|^{1-\epsilon_0}}|\partial f(z)|<\infty$ and $\limsup_{z\to\infty} \frac{1}{1+|z|^{D}} |\partial f(z)|<\infty$, 
where $\epsilon_{0}$ is given in Assumption \ref{Sampling design}. Then $f\in\mathcal{K}_3$.
\end{Lem}
\begin{proof}
Dividing the events and applying H\"older's inequality, we have
\begin{align*}
&\frac{1}{nh^2} \max_{1\leq j \leq n} \sup_{u \in (0,1)} E \left[|\partial f(\hat{\delta}_j+u(\delta_j-\hat{\delta}_j)|^2\right]\\
&=\frac{1}{nh^2} \max_{1\leq j \leq n} \sup_{u \in (0,1)} E \left[|\partial f(\hat{\delta}_j+u(\delta_j-\hat{\delta}_j))|^2; |\hat{\delta}_j+u(\delta_j-\hat{\delta}_j)|\leq1 \right ]\\
&{}\qquad+ \frac{1}{nh^2} \max_{1\leq j \leq n} \sup_{u \in (0,1)} E \left[|\partial f(\hat{\delta}_j+u(\delta_j-\hat{\delta}_j))|^2; |\hat{\delta}_j+u(\delta_j-\hat{\delta}_j)|>1 \right ]\\
&\lesssim\frac{1}{nh^2} \max_{1\leq j \leq n}E\left[\left|\hat{\delta}_j\right|^{2(1-\epsilon_0)}+\left|\delta_j-\hat{\delta}_j\right|^{2(1-\epsilon_0)}\right]\\
&{}\qquad+\frac{1}{nh^2} \max_{1\leq j \leq n} \left(E\left[1+\left|\hat{\delta}_j\right|^{\frac{2D} {\epsilon_0}}+\left|\delta_j-\hat{\delta}_j\right|^{\frac{2D} {\epsilon_0}}\right]\right)^{\epsilon_0}\left(P\left( \left|\hat{\delta}_j\right|
+\left|\delta_j-\hat{\delta}_j \right|>1\right)\right)^{1-\epsilon_0}\\
&\lesssim 
\frac{1}{nh^2}h^{1-\ep_{0}}=\frac{1}{nh^{1+\epsilon_0}}=o(1).
\end{align*}
The other condition can be verified as well.
\end{proof}

%\tcr{Put some details; refer Cor 3.8 for $z^3$ and $z^4$ to belong to the class.}
%First we note that the mappings $z\mapsto z^3,z^4$ satisfies Assumption \ref{Simple Assumption}, 
%hence by means of Proposition \ref{prop_sc-2.7} they are the elements of $\mathcal{K}_3$.
%We will prove this proposition below.
First we note that Assumption \ref{Moment fitting function}, the mappings $z\mapsto z^3,z^4,z\varphi(z),z^2\varphi(z)$ satisfies the conditions of Lemma \ref{Rate of gap} and Lemma \ref{K3}.
Let us show that for any $f\in\mathcal{K}_1\cap\mathcal{K}_3$ we have
\begin{equation}
\frac{1}{nh} \sum_jf(\hat{\delta}_j)\overset{P_0}\longrightarrow \nu_0(f).
\label{hm.add2.1}
\end{equation}
From a similar decomposition to that in the proof of Theorem \ref{Bias correction}, we have
\begin{align*}
&\frac{1}{nh} \sum_jf(\hat{\delta}_j)-\nu_0(f)\\
&=\left\{\frac{1}{nh} \sum_jf(\hat{\delta}_j)- \frac{1}{nh} \sum_jf(\delta_j)\right\}+ \left\{\frac{1}{nh} \sum_jf(\delta_j)- \frac{1}{nh} \sum_jf(\Delta_jJ)\right\}+\left\{\frac{1}{nh} \sum_jf(\Delta_jJ)-\nu_0(f)\right\}.
\end{align*} 
Then Lemma \ref{Rate of gap} implies the last term is $o_p(1)$.
Taylor's expansion and H\"older's inequality yield that
\begin{align*}
&\left|\frac{1}{nh} \sum_jf(\hat{\delta}_j)- \frac{1}{nh} \sum_jf(\delta_j)\right| =\left|\frac{1}{nh}\sum_j\int_0^1 f'(\hat{\delta}_j+u(\delta_j-\hat{\delta}_j))du(\delta_j-\hat{\delta}_j)\right|\\
&\leq\frac{1}{\sqrt{nh}} \frac{1}{nh}\sum_j\left|\int_0^1 f'(\hat{\delta}_j+u(\delta_j-\hat{\delta}_j))du\right|\left(\sup_\gamma \left|\partial c^{-1}_{j-1}(\gamma)\right|\right)\left|\Delta_jX\right||\hat{w}|\\
&+\frac{1}{\sqrt{nh}} \frac{1}{n}\sum_j\left|\int_0^1 f'(\hat{\delta}_j+u(\delta_j-\hat{\delta}_j))du\right|\left(\sup_\theta \left|\eta_{j-1}(\theta)\right|\right)|\hat{v}|\\
&\leq\sqrt{\frac{1}{(nh)^2}\sum_j\left|\int_0^1f'(\hat{\delta}_j+u(\delta_j-\hat{\delta}_j))du\right|^2} \times \sqrt{\frac{1}{nh}\sum_j\sup_\gamma \left|\partial c^{-1}_{j-1}(\gamma)\right|^2 \left|\Delta_jX\right|^2}\times O_{p}(1)
+o_p\bigg(\frac{1}{\sqrt{nh}}\bigg).
\end{align*}
Hence, using the conditioning argument together with $E[|\D_{j}X|^{2}]\lesssim h$ 
we obtain $ \frac{1}{nh} \sum_jf(\hat{\delta}_j)- \frac{1}{nh} \sum_jf(\delta_j)=o_p(1)$.
Recall that $E[|\Delta_jJ -\delta_j|^2]\lesssim h^2$, from which
\begin{align*}
&\left|\frac{1}{nh} \sum_jf(\delta_j)- \frac{1}{nh} \sum_jf(\Delta_jJ)\right|\\
&\leq\sqrt{\frac{1}{n} \sum_j\left|\int_0^1 f'(\Delta_jJ+u(\delta_j-\Delta_jJ))du\right|^2} \times \sqrt{\frac{1}{nh^2}\sum_j|\Delta_jJ-\delta_j|^2} = o_p(1),
\end{align*}
hence \eqref{hm.add2.1} follows. 

Now, \eqref{hm.add2.1} and Lemma \ref{K3} yields that $\hat{\Sigma}_{11,n}\overset{P_0}\longrightarrow\Sigma_{11}$. 
As in the proof of Lemma \ref{AN2}, it follows that
\begin{equation}\nn
\sup_{\theta\in\Theta}\left|\frac{1}{n}\sum_{j=1}^n\frac{\partial_{\alpha_k}a_{j-1}(\alpha)\partial_{\alpha_l}a_{j-1}(\alpha)}{c_{j-1}^2(\gamma)}-\int\frac{\partial_{\alpha_k} a(x,\alpha) \partial_{\alpha_l} a(x,\alpha)}{c^2(x,\gamma)} \pi_0(dx)\right|
\overset{P_0}\longrightarrow 0.
\end{equation}
Hence the consistency of $\hat{\theta}$ and the continuity of the map $\theta\mapsto\int\frac{\partial_{\alpha_k} a(x,\alpha) \partial_{\alpha_l} a(x,\alpha)}{c^2(x,\gamma)} \pi_0(dx)$ implies that
\begin{equation}\nn
\frac{1}{n}\sum_{j=1}^n\frac{\partial_{\alpha_k}\hat{a}_{j-1}\partial_{\alpha_l}\hat{a}_{j-1}}{\hat{c}_{j-1}^2}
\overset{P_0}\longrightarrow
\int\frac{\partial_{\alpha_k} a(x,\alpha_0) \partial_{\alpha_l} a(x,\alpha_0)}{c^2(x,\gamma_0)} \pi_0(dx).
\end{equation}
Similar estimates and Slutsky's lemma lead to $\hat{\Sigma}_{12,n}\overset{P_0}\longrightarrow\Sigma_{12}$ and $\hat{\Sigma}_{22,n}\overset{P_0}\longrightarrow\Sigma_{22}$.
The desired result follows from Theorem \ref{Bias correction}, Theorem \ref{JAN1} and Slutsky's lemma.

\subsection{Proof of Corollary \ref{Delta}} 
From the result of Theorem \ref{Bias correction}, $\frac{1}{nh_n}\sum_{j=1}^n\varphi(\hat{\delta}_j)-\nu_0(\varphi)=o_p(1)$.
Hence the continuity of $\partial F$ and the invertibility of $\partial F(\nu_0(\varphi), \theta_0)$ yield the first result.  
Finally, \cite[Theorem 3.1]{VanderVaart2000} leads to the second result.

%\subsection{Proof of Proposition \ref{prop_sc-2.7}}
%
%From the above estimates, we can conclude that $\vp\in\mathcal{K}_{1}\cap\mathcal{K}_{2}$ holds.
%
%Under Assumption \ref{Simple Assumption} we have $\limsup_{z\to 0}\frac{1}{|z|^{1-\epsilon_0}}|\partial f(z)|<\infty$ 
%and $\limsup_{z\to\infty} \frac{1}{1+|z|^{D}} |\partial f(z)|<\infty$ 
%for $f(z)$ being any of $z\varphi(z)$, $z^2 \varphi(z)$ and $\varphi_{1}\vp(z),\dots,\vp_{q}\vp(z)$. 
%This proves the latter claim of Proposition \ref{prop_sc-2.7}, thus completing the proof.

%% References
%\bibliography{/usr/local/texlive/2014/texmf-dist/bibtex/bib/biblatex/biblatex/List}

\bigskip

\begin{thebibliography}{99}

\bibitem{Applebaum2009}
Applebaum, D. (2009).  {\it L{\'e}vy processes and stochastic calculus}:
  Cambridge university press.

\bibitem{yuima}
Brouste, A., Fukasawa, M., Hino, H., Iacus, S, Kamatani, K., Koike, Y., Masuda, H., Nomura, R., Ogihara, T., Shimuzu, Y., Uchida, M., Yoshida, N. (2014), The YUIMA project: A computational framework for simulation and inference of stochastic differential equations. 
{\it Journal of Statistical Software}, 57, 1--51.
      
%\bibitem{ComGen11}
%Comte, F. and Genon-Catalot, V. (2011).
%Estimation for L\'{e}vy processes from high frequency data within a long time interval. 
%{\it Ann. Statist.}, 39, 803--837.

%\bibitem{ComGen15}
%Comte, F. and Genon-Catalot, V. (20156r).
%Adaptive Estimation for L\'{e}vy Processes
%Levy Matters IV, Estimation for Discretely Observed Levy Processes, pp.77--177, Lecture Notes in Mathematics, Vol. 2128 (2015), Springer.

\bibitem{Dvoretzky1972}
Dvoretzky, A. (1972). Asymptotic normality for sums of dependent random
  variables. in  {\it Proceedings of the {S}ixth {B}erkeley {S}ymposium on
  {M}athematical {S}tatistics and {P}robability ({U}niv. {C}alifornia,
  {B}erkeley, {C}alif., 1970/1971), {V}ol. {II}: {P}robability theory},
  513--535: Univ. California Press, Berkeley, Calif.

\bibitem{Feu90}
Feuerverger, A. (1990).
An efficiency result for the empirical characteristic function in stationary time-series models. 
{\it Canad. J. Statist.}, 18, 155--161. 

\bibitem{Lopez2008}
Figueroa-L{\'o}pez, J.~E. (2008). Small-time moment asymptotics for L{\'e}vy
  processes. {\it Statistics \& Probability Letters}, 78,  3355--3365.

\bibitem{Lopez2009}
Figueroa-L{\'o}pez, J.~E. (2009). Nonparametric estimation for L{\'e}vy models based on
  discrete-sampling. {\it Lecture Notes-Monograph Series},  117--146.

\bibitem{Genon-catalot1993}
Genon-Catalot, V. and Jacod, J. (1993). On the estimation of the diffusion
  coefficient for multi-dimensional diffusion processes. in  {\it Annales de
  l'institut Henri Poincar{\'e} (B) Probabilit{\'e}s et Statistiques}, 29,
  119--151, Gauthier-Villars.

\bibitem{Gobet2002}
Gobet, E. (2002). LAN property for ergodic diffusions with discrete
  observations. {\it Annales de l'Institut Henri Poincare (B) Probability
  and Statistics}, 38,  711--737.

%\bibitem{Nina_thesis}
%Jakobsen, N.~M. (2015). Efficient estimating functions for stochastic differential equations. 
%PhD thesis, submitted. 
%Available at {\tt http://www.math.ku.dk/noter/filer/phd15nmj.pdf}

\bibitem{Jac07}
Jacod, J. (2007). Asymptotic properties of power variations of L\'{e}vy processes. 
{\it ESAIM Probab. Stat.}, 11, 173--196. 

\bibitem{Kessler1997}
Kessler, M. (1997). Estimation of an ergodic diffusion from discrete
  observations. {\it Scandinavian Journal of Statistics}, 24,  211--229.

\bibitem{Kunita1997}
Kunita, H. (1997).  {\it Stochastic flows and stochastic differential
  equations}, 24: Cambridge university press.

\bibitem{Kutoyants2004}
Kutoyants, Y.~A. (2004).  {\it Statistical inference for ergodic diffusion
  processes}: Springer Science \& Business Media.

\bibitem{Liptser2001}
Liptser, R.~S., Shiryaev, A.~N. 
  (2001).  {\it Statistics of Random Processes II: II. Applications}, 2:
  Springer Science \& Business Media.

\bibitem{LusPag08}
Luschgy, H. and Pag\`{e}s, G. (2008). 
Moment estimates for L\'{e}vy processes. 
{\it Electron. Commun. Probab.}, 13, 422--434. 

\bibitem{Masuda2013}
Masuda, H. (2013). Convergence of Gaussian quasi-likelihood random
  fields for ergodic L{\'e}vy driven SDE observed at high frequency. {\it The
  Annals of Statistics}, 41,  1593--1641.

%\bibitem{Mer07}
%Merkouris, T. (2007).
%Transform martingale estimating functions. {\it Ann. Statist.}, 35, 1975--2000.

%\bibitem{Pra87}
%Prakasa Rao, B.~L.~S. (1987). 
%{\it Asymptotic theory of statistical inference.} 
%John Wiley \& Sons, Inc., New York.

\bibitem{rct}
R Development Core Team: R (2010), A language and environment for statistical computing. 
R Foundation for Statistical Computing, Vienna, Austria. %http://www.R-project.org.

\bibitem{Rao1999}
Prasaka Rao, B.(1999).  {\it Statistical inference for diffusion type processes}:
  Arnold.

%\bibitem{Protter2004}
%Protter, P.~E. (2004).  {\it Stochastic Integration and Differential Equations:
%  Version 2.1}, 21: Springer Science \& Business Media.

\bibitem{Sato1999}
Sato, K.-i. (1999).  {\it L{\'e}vy processes and infinitely divisible
  distributions}: Cambridge university press.

\bibitem{Shimizu2009}
Shimizu, Y. (2009). Functional estimation for Levy measures of semimartingales
  with Poissonian jumps. {\it Journal of Multivariate Analysis}, 100,
  1073--1092.

\bibitem{VanderVaart2000}
van~der Vaart, A.~W. (2000).  {\it Asymptotic statistics}. Cambridge
  university press.

%\bibitem{Yoshida2011}
%Yoshida, N. (2011). Polynomial type large deviation inequalities and
%  quasi-likelihood analysis for stochastic differential equations. {\it Annals
%  of the Institute of Statistical Mathematics}, 63,  431--479.


\end{thebibliography}

%%%%%%%%%%%%%%%%%%%%%%%%%%%%%%%%%%%%%%%%%%%%%%%%%%%%%%%%%%%%%%%%%%%
%
%  This bbl file is created by jecon.bst ver.3.1
%  The latest jecon.bst is available at
%  <http://shirotakeda.org/home-ja/tex-ja/jecon-ja.html>
%
%%%%%%%%%%%%%%%%%%%%%%%%%%%%%%%%%%%%%%%%%%%%%%%%%%%%%%%%%%%%%%%%%%%

\bigskip

\subsection*{Acknowledgement}
We are grateful to the referees for careful reading and constructive comments, 
which led to substantial improvements of the earlier version of this paper. 
This work was partly supported by JSPS KAKENHI Grant Numbers 26400204 (H. Masuda) and 
JST, CREST.% 744 Motooka Nishi-ku Fukuoka 819-0395, Japan

\end{document}